\documentclass[11pt,reqno,a4paper]{amsart}

\usepackage{graphicx}
\usepackage{amsmath}
\usepackage{amssymb}
\usepackage{amsfonts}
\usepackage{color}
\allowdisplaybreaks

\newtheorem{theorem}{Theorem}[section]
\newtheorem{lemma}[theorem]{Lemma}

\newtheorem{proposition}[theorem]{Proposition}

\theoremstyle{definition}
\newtheorem{remark}[theorem]{Remark}

\begin{document}

\author[M.~Kuba]{Markus Kuba}
\address{Markus Kuba\\
Institute of Applied Mathematics and Natural Sciences\\
University of Applied Sciences - Technikum Wien\\
H\"ochst\"adtplatz 5, 1200 Wien} %
\email{kuba@technikum-wien.at}

\author[A.~Panholzer]{Alois Panholzer}
\address{Alois Panholzer\\
Institut f{\"u}r Diskrete Mathematik und Geometrie\\
Technische Universit\"at Wien\\
Wiedner Hauptstr. 8-10/104\\
1040 Wien, Austria} \email{Alois.Panholzer@tuwien.ac.at}

\thanks{The second author was supported by the Austrian Science Foundation FWF, grant P25337-N23.}

\title{Combinatorial analysis of growth models for series-parallel networks}

\begin{abstract}
We give combinatorial descriptions of two stochastic growth models for series-parallel networks introduced by Hosam Mahmoud by encoding the growth process via recursive tree structures. Using decompositions of the tree structures and applying analytic combinatorics methods allows a study of quantities in the corresponding series-parallel networks. For both models we obtain limiting distribution results for the degree of the poles and the length of a random source-to-sink path, and furthermore we get asymptotic results for the expected number of source-to-sink paths. Moreover, we introduce generalizations of these stochastic models by encoding the growth process of the networks via further important increasing tree structures and give an analysis of some parameters.
\end{abstract}

\maketitle

\section{Introduction}

Series-parallel networks are two-terminal graphs, i.e., they have two distinguished vertices called the source and the sink, that can be constructed recursively by applying two simple composition operations, namely the parallel composition (where the sources and the sinks of two series-parallel networks are merged) and the series composition (where the sink of one series-parallel network is merged with the source of another series-parallel network). Here we will always consider series-parallel networks as digraphs with edges oriented in direction from the north-pole, the source, towards the south-pole, the sink. Such graphs can be used to model the flow in a bipolar network, e.g., of current in an electric circuit or goods from the producer to a market. Furthermore series-parallel networks and series-parallel graphs (i.e., graphs which are series-parallel networks when some two of its vertices are regarded as source and sink; see, e.g., \cite{BraLeSpi1999} for exact definitions and alternative characterizations) are of interest in computational complexity theory, since some in general NP-complete graph problems are solvable in linear time on series-parallel graphs (e.g., finding a maximum independent set).

Recently there occurred several studies concerning the typical behaviour of structural quantities (as, e.g., node-degrees, see~\cite{DrmGimNoy2010}) in series-parallel graphs and networks under a uniform model of randomness, i.e., where all series-parallel graphs of a certain size (counted by the number of edges) are equally likely. In contrast to these uniform models, Mahmoud~\cite{Mahmoud2013,Mahmoud2014} introduced two interesting growth models for series-parallel networks, which are generated by starting with a single directed arc from the source to the sink and iteratively carrying out serial and parallel edge-duplications according to a stochastic growth rule; we call them uniform Bernoulli edge-duplication rule (``Bernoulli model'' for short) and uniform binary saturation edge-duplication rule (``binary model'' for short). A formal description of these models is given in Section~\ref{sec:Growth_models_recursive_trees}. Using the defining stochastic growth rules and a description via P\'{o}lya-Eggenberger urn models (see, e.g., \cite{Mahmoud2009}), several quantities for series-parallel networks (as the number of nodes of small degree and the degree of the source for the Bernoulli model, and the length of a random source-to-sink path for the binary model) are treated in \cite{Mahmoud2013,Mahmoud2014}.

The aim of this work is to give an alternative description of these growth models for series-parallel networks by encoding the growth of them via recursive tree structures, to be precise, via edge-coloured recursive trees and so-called bucket recursive trees (see \cite{KubPan2010} and references therein). The advantage of such a modelling is that these objects allow not only a stochastic description (the tree evolution process which reflects the growth rule of the series-parallel network), but also a combinatorial one (as certain increasingly labelled trees or bucket trees), which gives rise to a top-down decomposition of the structure. An important observation is that indeed various interesting quantities for series-parallel networks can be studied by considering certain parameters in the corresponding recursive tree model and making use of the combinatorial decomposition. We focus here on the quantities degree $D_{n}$ of the source and/or sink, length $L_{n}$ of a random source-to-sink path and the number $P_{n}$ of source-to-sink paths in a random series-parallel network of size $n$, but mention that also other quantities (as, e.g., the number of ancestors, node-degrees, or the number of paths through a random or the $j$-th edge) could be treated in a similar way. By using analytic combinatorics techniques (see~\cite{FlaSed2009}) we obtain limiting distribution results for $D_{n}$ and $L_{n}$ (thus answering questions left open in \cite{Mahmoud2013,Mahmoud2014}), whereas for the random variable (r.v.\ for short) $P_{n}$ (whose distributional treatment seems to be considerably more involved) we are able to give asymptotic results for the expectation. These results and their derivations are given in Section~\ref{sec:Bernoulli_model} and Section~\ref{sec:Binary_model} for the Bernoulli model and for the binary model, respectively. The combinatorial approach presented is flexible enough to allow also a study of series-parallel networks generated by modifications of the presented edge-duplication rules. This is illustrated in Section~\ref{sec:Bernoulli_non-uniform_rules}, where two Bernoulli models with a non-uniform edge-duplication rule and combinatorial descriptions via certain edge-coloured increasing trees are introduced, as well as in Section~\ref{sec:bary_saturation_rules}, where a $b$-ary saturation model and its encoding via edge-coloured bucket recursive trees with bucket size $b \ge 2$ is proposed.

Mathematically, an analytic combinatorics treatment of the quantities of interest leads to studies of first and second order non-linear differential equations. In this context we want to mention that another model of series-parallel networks called increasing diamonds has been introduced recently in \cite{BodDieFonGenHWa2015+}. A treatment of quantities in such networks inherently also yields a study of second order non-linear differential equations; however, the definition as well as the structure of increasing diamonds is quite different from the models treated here as can be seen also by comparing the behaviour of typical graph parameters (e.g., the number of source-to-sink paths $P_{n}$ in increasing diamonds is trivially bounded by $n$, whereas in the models studied here the expected number of paths grows exponentially). We mention that the analysis of the structures considered here has further relations to other objects; e.g., it holds that the Mittag-Leffler limiting distributions occurring in Theorem~\ref{thm:Bernoulli_degree_limit} \& \ref{thm:Bernoulli_length_path} also appear in other combinatorial contexts as in certain triangular balanced urn models (see \cite{Janson2010}) or implicitly in the recent study of an extra clustering model for animal grouping \cite{DrmFucLee2015+} (after scaling, as continuous part of the characterization given in \cite[Theorem~2]{DrmFucLee2015+}, since it is possible to simplify some of the representations given there). Also the characterizations of the limiting distribution for $D_{n}$ and $L_{n}$ of binary series-parallel networks via the sequence of $r$-th integer moments satisfies a recurrence relation of ``convolution type'' similar to ones occurring in \cite{CheFerHwaMar2014}, for which asymptotic studies have been carried out.
Furthermore, the described top-down decomposition of the combinatorial objects makes these structures amenable to other methods, in particular, it seems that the contraction method, see, e.g., \cite{NeiRue2004,NeiSul2015}, allows an alternative characterization of limiting distributions occurring in the analysis of binary series-parallel networks.

\section{Series-parallel networks and description via recursive tree structures\label{sec:Growth_models_recursive_trees}}

\subsection{Bernoulli model\label{ssec:Bernoulli}}

In the Bernoulli model in step $1$ one starts with a single edge labelled $1$ connecting the source and the sink, and in step $n$, with $n > 1$, one of the $n-1$ edges of the already generated series-parallel network is chosen uniformly at random, let us assume it is edge $j=(x,y)$; then either with probability $p$, $0 < p < 1$, this edge is doubled in a parallel way\footnote{In the original work \cite{Mahmoud2013} the r\^oles of $p$ and $q$ are switched, but we find it catchier to use $p$ for the probability of a parallel doubling.}, i.e., an additional edge $(x,y)$ labelled $n$ is inserted into the graph (let us say, right to edge $e$), or otherwise, thus with probability $q := 1-p$, this edge is doubled in a serial way, i.e., edge $(x,y)$ is replaced by the series of edges $(x,z)$ and $(z,y)$, with $z$ a new node, where $(x,z)$ gets the label $j$ and $(z,y)$ will be labelled by $n$.

The growth of series-parallel networks corresponds with the growth of random recursive trees, where one starts in step $1$ with a node labelled $1$, and in step $n$ one of the $n-1$ nodes is chosen uniformly at random and node $n$ is attached to it as a new child. Thus, a doubling of edge $j$ in step $n$ when generating the series-parallel network corresponds in the recursive tree to an attachment of node $n$ to node $j$. Additionally, in order to keep the information about the kind of duplication of the chosen edge, the edge incident to $n$ is coloured either blue encoding a parallel doubling, or coloured red encoding a serial doubling. Such combinatorial objects of edge-coloured recursive trees can be described via the formal equation
\begin{equation*}
  \mathcal{T} = \mathcal{Z}^{\Box} \ast \text{\textsc{SET}}(\{B\} \times \mathcal{T} + \{R\} \times \mathcal{T}),
\end{equation*}
with $B$ and $R$ markers (see \cite{FlaSed2009}). Of course, one has to keep track of the number of blue and red edges to get the correct probability model according to 
\begin{equation*}
  \mathbb{P}\{\text{$T \in \mathcal{T}_{n}$ is chosen}\} = \frac{p^{\# \text{blue edges of $T$}} \cdot q^{\# \text{red edges of $T$}}}{T_{n}},
\end{equation*}
where $\mathcal{T}_{n} = \{T \in \mathcal{T} : \text{$T$ has order $n$}\}$ and $T_{n} = (n-1)!$ the number of different (uncoloured) recursive trees of order $n$. Throughout this work the term order of a tree $T$ shall denote the number of labels contained in $T$, which, of course, for recursive trees coincides with the number of nodes of $T$. Then, each edge-coloured recursive tree of order $n$ and the corresponding series-parallel network of size $n$ occur with the same probability. Combinatorially, to get the right probability model we will assume that each marker $B$ gets the multiplicative weight $p$ and each marker $R$ the weight $q=1-p$. An example for a series-parallel network grown via the Bernoulli model and the corresponding edge-coloured recursive tree is given in Figure~\ref{fig:growth_Bernoulli}.
Note that per se, according to the growth rule, in the structures considered (i.e., series-parallel network models and recursive tree models) there is no ordering on the children of a node, but we always assume canonical plane representations of these non-plane objects based on an order left-to-right given by the integer order of the labels of the ``attracted edges''.
\begin{figure}
\begin{center}
\begin{minipage}{13cm}
\includegraphics[height=2.5cm]{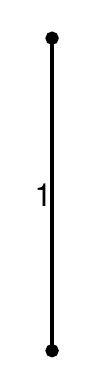}
\includegraphics[height=2.5cm]{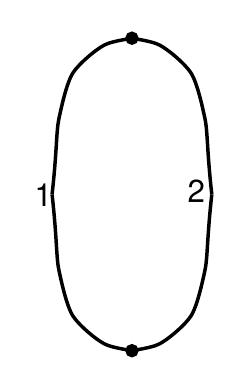}
\includegraphics[height=2.5cm]{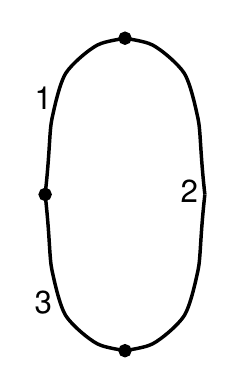}
\includegraphics[height=2.5cm]{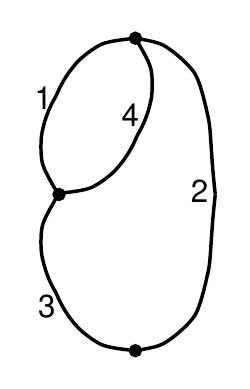}
\includegraphics[height=2.5cm]{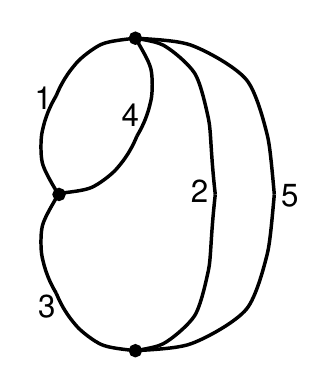}
\includegraphics[height=2.5cm]{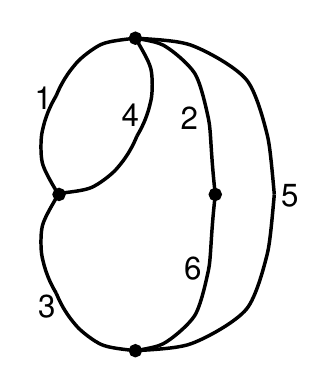}
\includegraphics[height=2.5cm]{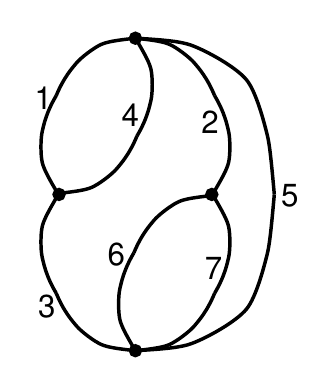}\hfill
\\
\includegraphics[height=2.5cm]{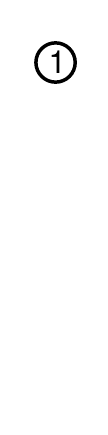}
\hspace*{5mm}\includegraphics[height=2.5cm]{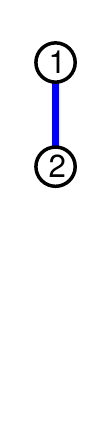}
\hspace*{5mm}\includegraphics[height=2.5cm]{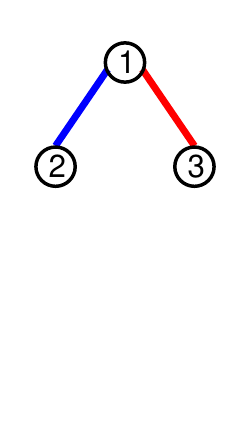}
\includegraphics[height=2.5cm]{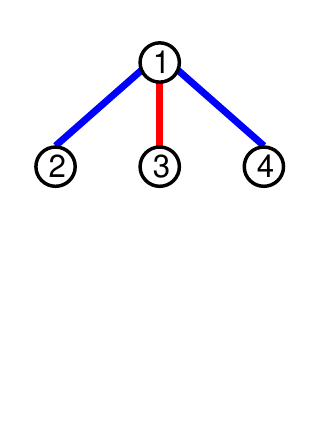}
\includegraphics[height=2.5cm]{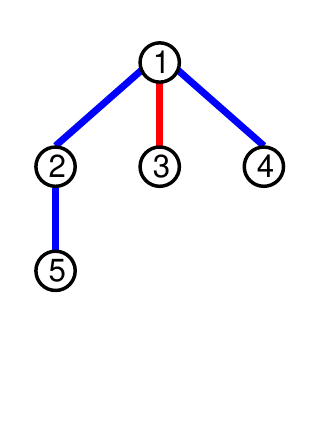}
\includegraphics[height=2.5cm]{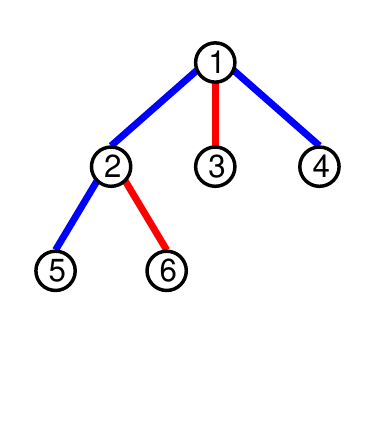}
\includegraphics[height=2.5cm]{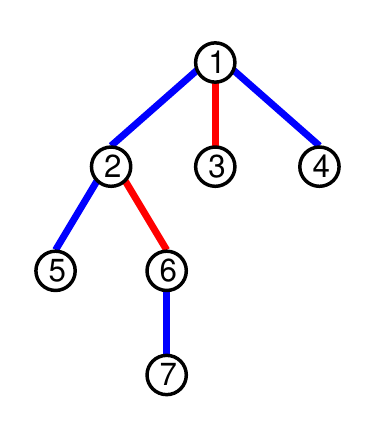}\hfill
\end{minipage}
\end{center}
\vspace*{-6mm}
\caption{Growth of a series-parallel network under the Bernoulli model and of the corresponding edge-coloured recursive tree. In the resulting graph the degree of the source is $4$, the length of the leftmost source-to-sink path is $2$ and there are $5$ different source-to-sink paths.\label{fig:growth_Bernoulli}}
\end{figure}

\subsection{Binary model\label{ssec:Binary}}

In the binary model again in step $1$ one starts with a single edge labelled $1$ connecting the source and the sink, and in step $n$, with $n>1$ one of the $n-1$ edges of the already generated series-parallel network is chosen uniformly at random; let us assume it is edge $j=(x,y)$; but now whether edge $j$ is doubled in a parallel or serial way is already determined by the out-degree of node $x$: if node $x$ has out-degree $1$ then we carry out a parallel doubling by inserting an additional edge $(x,y)$ labelled $n$ into the graph right to edge $j$, but otherwise, i.e., if node $x$ has out-degree $2$ and is thus already saturated, then we carry out a serial doubling by replacing edge $(x,y)$ by the edges $(x,z)$ and $(z,y)$, with $z$ a new node, where $(x,z)$ gets the label $j$ and $(z,y)$ will be labelled by $n$.

It turns out that the growth model for binary series-parallel networks corresponds with the growth model for bucket recursive trees \cite{MahSmy1995} with maximal bucket size $2$, i.e., where nodes in the tree can hold up to two labels: in step $1$ one starts with the root node containing label $1$, and in step $n$ one of the $n-1$ labels in the tree is chosen uniformly at random, let us assume it is label $j$, and attracts the new label $n$. If the node $x$ containing label $j$ is saturated, i.e., it contains already two labels, then a new node containing label $n$ will be attached to $x$ as a new child. Otherwise, label $n$ will be inserted into node $x$; now, $x$ contains the labels $j$ and $n$. As has been pointed out in \cite{KubPan2010} such random bucket recursive trees can also be described in a combinatorial way by extending the notion of increasing trees: namely a bucket recursive tree is either a node labelled $1$ or it consists of the root node labelled $(1,2)$, where two (possibly empty) forests of (suitably relabelled) bucket recursive trees are attached to the root as a left forest and a right forest. A formal description of the family $\mathcal{B}$ of bucket recursive trees (with bucket size at most $2$) is in modern notation given as follows:
\begin{equation*}
  \mathcal{B} = \mathcal{Z}^{\Box} + \mathcal{Z}^{\Box} \ast \left(\mathcal{Z}^{\Box} \ast \left(\text{\textsc{SET}}(\mathcal{B}) \ast \text{\textsc{SET}}(\mathcal{B})\right)\right).
\end{equation*}
It follows from this formal description that there are $T_{n} = (n-1)!$ different bucket recursive trees with $n$ labels, i.e., of order $n$, and furthermore it has been shown in \cite{KubPan2010} that this combinatorial description (assuming the uniform model, where each of these trees occurs with the same probability) indeed corresponds with the stochastic description of random bucket recursive trees of order $n$ given before. An example for a binary series-parallel network and the corresponding bucket recursive tree is given in Figure~\ref{fig:growth_binary}.
\begin{figure}
\begin{center}
\begin{minipage}{14cm}
\includegraphics[height=2.5cm]{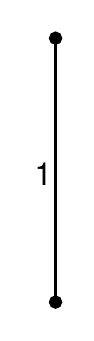}
\includegraphics[height=2.5cm]{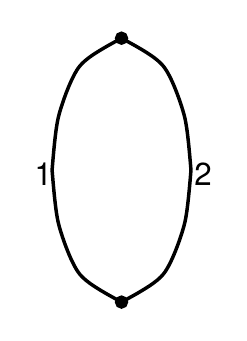}
\includegraphics[height=2.5cm]{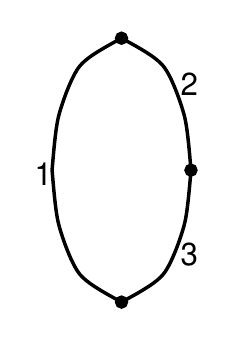}
\includegraphics[height=2.5cm]{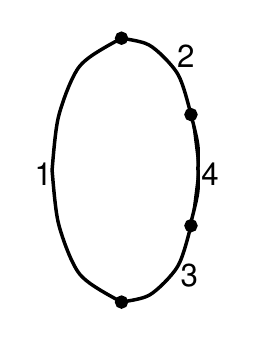}
\includegraphics[height=2.5cm]{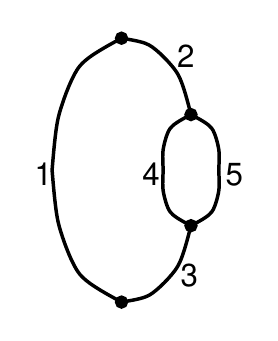}
\includegraphics[height=2.5cm]{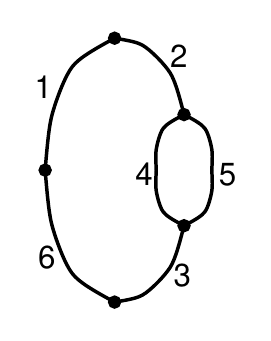}
\hspace*{4mm}\includegraphics[height=2.5cm]{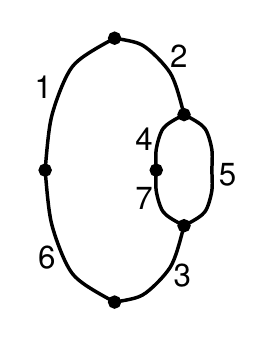}\hfill
\\
\includegraphics[height=2.5cm]{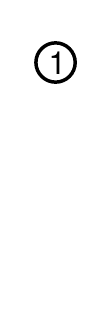}
\hspace*{3mm}\includegraphics[height=2.5cm]{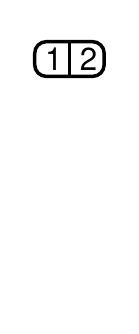}
\hspace*{6mm}\includegraphics[height=2.5cm]{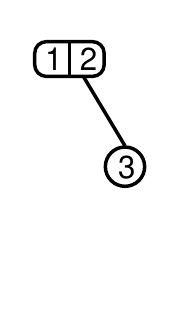}
\hspace*{3mm}\includegraphics[height=2.5cm]{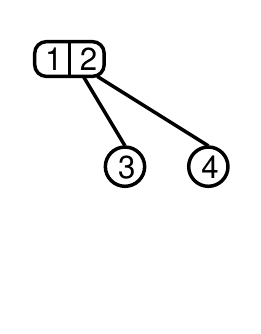}
\hspace*{-3mm}\includegraphics[height=2.5cm]{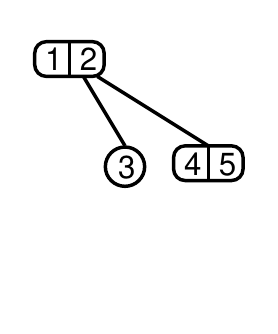}
\hspace*{-3mm}\includegraphics[height=2.5cm]{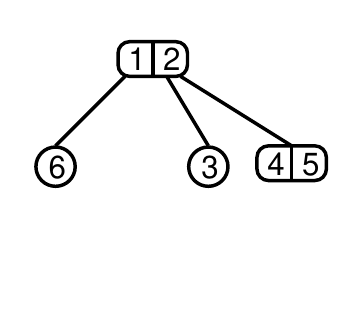}
\hspace*{-3mm}\includegraphics[height=2.5cm]{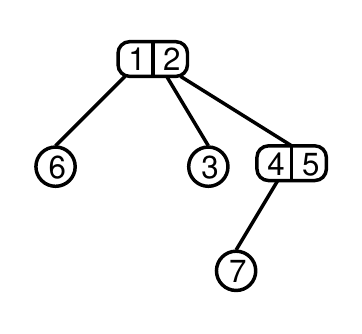}\hfill
\end{minipage}
\end{center}
\vspace*{-6mm}
\caption{Growth of a binary series-parallel network and of the corresponding bucket recursive tree. In the resulting graph the degree of the sink is $2$, the length of the leftmost source-to-sink path is $2$ and there are $3$ different source-to-sink paths.\label{fig:growth_binary}}
\end{figure}

In our analysis of binary series-parallel networks the following link between the decomposition of a bucket recursive tree $T$ into its root $(1,2)$ and the left forest (consisting of the trees $T_{1}^{[L]}, \dots, T_{\ell}^{[L]}$) and the right forest (consisting of the trees $T_{1}^{[R]}, \dots, T_{r}^{[R]}$), and the subblock structure of the corresponding binary network $G$ is important: $G$ consists of a left half $G^{[L]}$ and a right half $G^{[R]}$ (which share the source and the sink), where $G^{[L]}$ is formed by a series of blocks (i.e., maximal $2$-connected components) consisting of the edge labelled $1$ followed by binary networks corresponding to $T_{\ell}^{[L]}$, $T_{\ell-1}^{[L]}$, \dots, $T_{1}^{[L]}$, and $G^{[R]}$ is formed by a series of blocks consisting of the edge labelled $2$ followed by binary networks corresponding to $T_{r}^{[R]}$, $T_{r-1}^{[R]}$, \dots, $T_{1}^{[R]}$; see Figure~\ref{fig:link_decomposition_buckettree_network} for an example.
\begin{figure}
\begin{center}
\parbox{6.2cm}{\includegraphics[width=6cm]{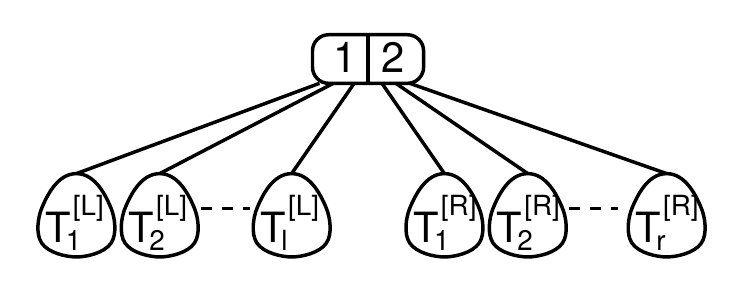}} \quad $\Longleftrightarrow$ \quad
\parbox{2.7cm}{\includegraphics[width=2.5cm]{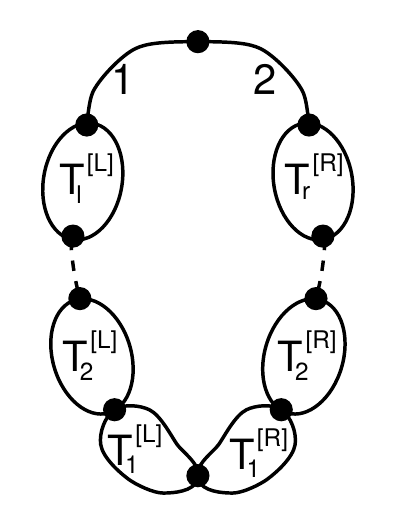}}
\end{center}
\vspace*{-3mm}
\caption{Decomposition of a bucket recursive tree $T$ into its root and the left and right forest, respectively, and the subblock structure of the corresponding binary network.\label{fig:link_decomposition_buckettree_network}}
\end{figure}

\section{Uniform Bernoulli edge-duplication growth model\label{sec:Bernoulli_model}}

\subsection{Degree of the source\label{ssec:Bernoulli_Degree}}

Let $D_{n}=D_{n}(p)$ denote the r.v.\ measuring the degree of the source in a random series-parallel network of size $n$ for the Bernoulli model, with $0 < p < 1$. A first analysis of this quantity has been given in \cite{Mahmoud2013}, where the exact distribution of $D_{n}$ as well as exact and asymptotic results for the expectation $\mathbb{E}(D_{n})$ could be obtained. However, questions concerning the limiting behaviour of $D_{n}$ and the asymptotic behaviour of higher moments of $D_{n}$ have not been touched; in this context we remark that the explicit results for the probabilities $\mathbb{P}\{D_{n}=m\}$ as obtained in \cite{Mahmoud2013} and restated in Theorem~\ref{thm:Bernoulli_degree_limit} are not easily amenable to asymptotic studies, because of large cancellations of the alternating summands in the corresponding formula. We will reconsider this problem by applying the combinatorial approach introduced in Section~\ref{sec:Growth_models_recursive_trees}, and in order to get the limiting distribution we apply methods from analytic combinatorics. As has been already remarked in \cite{Mahmoud2013} the degree of the sink is equally distributed as $D_{n}$ due to symmetry reasons, although a simple justification of this fact via direct ``symmetry arguments'' does not seem to be completely trivial (the insertion process itself is a priori not symmetric w.r.t.\ the poles, since edges are always inserted towards the sink); however, it is not difficult to show this equality by establishing and treating a recurrence for the distribution of the sink, which is here omitted.

To state the theorem we define (see \cite{Janson2010}, where also relations to stable random variables are given) a r.v.\ $Y_{p} \stackrel{(d)}{=} \text{Mittag-Leffler}(p)$ to be Mittag-Leffler distributed with parameter $p$ when its $r$-th integer moments are given as follows:
\begin{equation*}
  \mathbb{E}(Y_{p}^{r}) = \frac{r!}{\Gamma(rp+1)}, \quad \text{for $r \ge 0$}.
\end{equation*}
The distribution of $Y_{p}$ can also be characterized via its density function $f(x)$, which can be computed, e.g., from the moment generating function $M(z) = \mathbb{E}(e^{z Y_{p}}) = \sum_{r \ge 0} \mathbb{E}\big(Y_{p}^{r}\big) \frac{z^{r}}{r!}$ by applying the inverse Laplace transform:
\begin{equation*}
  f(x) = \frac{1}{2 \pi i} \int_{\mathcal{H}} \frac{e^{-t - x (-t)^{p}}}{(-t)^{1-p}} dt, \quad \text{for $x > 0$},
\end{equation*}
with $\mathcal{H}$ a Hankel contour starting from $e^{2 \pi i} \infty$, passing around $0$ and terminating at $+ \infty$.
We remark that after simple manipulations $f(x)$ can also be written as the following real integral:
\begin{equation*}
  f(x) = \frac{1}{\pi p} \int_{0}^{\infty} e^{-w^{\frac{1}{p}} - xw \cos(\pi p)} \sin(\pi p - x w \sin(\pi p)) dw, \quad \text{for $x>0$}.
\end{equation*}
We further use throughout this work the abbreviations $x^{\underline{k}} := x \cdot (x-1) \cdots (x-k+1)$ and $x^{\overline{k}} := x \cdot (x+1) \cdots (x+k-1)$ for the falling and rising factorials, respectively.

\begin{theorem}\label{thm:Bernoulli_degree_limit}
The degree $D_{n}$ of the source or the sink in a randomly chosen series-parallel network of size $n$ generated by the Bernoulli model
has the following probability distribution:
\begin{equation}\label{eqn:exact_dist_Dn}
  \mathbb{P}\{D_{n}=m\} = \sum_{j=0}^{m-1} \binom{m-1}{j} (-1)^{n+j-1} \binom{p(j+1)-1}{n-1}, \quad \text{for $1 \le m \le n$}.
\end{equation}

Moreover, $D_{n}$ converges after scaling, for $n \to \infty$, in distribution to a Mittag-Leffler distributed r.v.\ with parameter $p$:
\begin{equation*}
  \frac{D_{n}}{n^{p}} \xrightarrow{(d)} D = D(p), \quad \text{with} \quad D \stackrel{(d)}{=} \emph{\text{Mittag-Leffler}}(p).
\end{equation*}
\end{theorem}
\begin{remark}
For the particular instance $p=\frac{1}{2}$ one can evaluate the Hankel contour integral occurring above and obtains that the limiting distribution $D$ is characterized by the density function $f(x) = \frac{1}{\sqrt{\pi}} \cdot e^{-\frac{x^{2}}{4}}$, for $x > 0$. Thus, for $p=\frac{1}{2}$, $f(x)$ is the density function of a so-called half-normal distribution with parameter $\sigma = \sqrt{2}$.
\end{remark}
\begin{proof}
When considering the description of the growth process of these series-parallel networks via edge-coloured recursive trees it is apparent that the degree of the source in such a graph corresponds to the order of the maximal subtree containing the root node and only blue edges, i.e., we have to count the number of nodes in the recursive tree that can be reached from the root node by taking only blue edges; for simplicity we denote this maximal subtree by ``blue subtree''. Thus, in the recursive tree model, $D_{n}$ measures the order of the blue subtree in a random edge-coloured recursive tree of order $n$. To treat $D_{n}$ we introduce the r.v.\ $D_{n,k}$, whose distribution is given as the conditional distribution $D_{n} \big| \{\text{the tree has exactly $k$ blue edges}\}$, and the trivariate generating function
\begin{equation*}
  F(z,u,v) := \sum_{n \ge 1} \sum_{0 \le k \le n-1} \sum_{m \ge 1} T_{n} \binom{n-1}{k} \mathbb{P}\{D_{n,k} = m\} \frac{z^{n}}{n!} u^{k} v^{m},
\end{equation*}
with $T_{n} = (n-1)!$ the number of recursive trees of order $n$. Thus $T_{n} \binom{n-1}{k} \mathbb{P}\{D_{n,k} = m\}$ counts the number of edge-coloured recursive trees of order $n$ with exactly $k$ blue edges, where the blue subtree has order $m$. Additionally we introduce the auxiliary function $N(z,u) := \sum_{n \ge 1} \sum_{0 \le k \le n-1} T_{n} \binom{n-1}{k} \frac{z^{n}}{n!} u^{k} = \frac{1}{1+u} \log\left(\frac{1}{1-z(1+u)}\right)$, i.e., the exponential generating function of the number of edge-coloured recursive trees of order $n$ with exactly $k$ blue edges.

The decomposition of a recursive tree into its root node and the set of branches attached to it immediately can be translated into a differential equation for $F(z,u,v)$, where we only have to take into account that the order of the blue subtree in the whole tree is one (due to the root node) plus the orders of the blue subtrees of the branches which are connected to the root node by a blue edge (i.e., only branches which are connected to the root node by a blue edge will contribute). Namely, with $F:= F(z,u,v)$ and $N := N(z,u)$, we get the first order separable differential equation
\begin{equation}\label{eqn:DEQ_source-degree}
  F' = v \cdot e^{u F + N},
\end{equation}
with initial condition $F(0,u,v) = 0$. Throughout this work, the notation $f'$ for (multivariate) functions $f(z, \dots)$ shall always denote the derivative w.r.t.\ the variable $z$. The exact solution of \eqref{eqn:DEQ_source-degree} can be obtained by standard means and is given as follows:
\begin{equation}\label{eqn:Bernoulli_Fzuv_sol}
  F(z,u,v) = \frac{1}{u} \log \left(\frac{1}{1-v+v(1-z(1+u))^{\frac{u}{1+u}}}\right).
\end{equation}

Since we are only interested in the distribution of $D_{n}$ we will actually consider the generating function
\begin{equation*}
  F(z,v) := \sum_{n \ge 1} \sum_{m \ge 1} T_{n} \mathbb{P}\{D_{n} = m\} \frac{z^{n}}{n!} v^{m} = \sum_{n \ge 1} \sum_{m \ge 1} \mathbb{P}\{D_{n} = m\} \frac{z^{n}}{n} v^{m}.
\end{equation*}
Now, according to the definition of the conditional r.v.\ $D_{n,k}$, we have\newline
$\mathbb{P}\{D_{n}=m\} = \sum_{k=0}^{n-1} \mathbb{P}\{D_{n,k}=m\} \binom{n-1}{k} p^{k} q^{n-1-k}$, which, after simple computations, gives the relation 
\begin{equation*}
  F(z,v) = \frac{1}{q} F\big(qz, \frac{p}{q}, v\big).
\end{equation*}
Thus, we obtain from \eqref{eqn:Bernoulli_Fzuv_sol} the following explicit formula for $F'(z,v)$, which has been obtained already in \cite{Mahmoud2013} by using a description of $D_{n}$ via urn models:
\begin{equation}\label{eqn:gf_sourcedeg_explicit}
  F'(z,v) = \frac{v}{v(1-z) + (1-v) (1-z)^{1-p}}.
\end{equation}
Extracting coefficients from \eqref{eqn:gf_sourcedeg_explicit} immediately yields via $\mathbb{P}\{D_{n}=m\} = [z^{n-1} v^{m}] F'(z,v)$ the explicit result for the probability distribution of $D_{n}$ obtained by Mahmoud \cite{Mahmoud2013} and restated above.

In order to describe the limiting distribution of $D_{n}$ we study the integer moments. To do this we introduce $\tilde{F}(z,w) := F(z,1+w)$, since we get for its derivative the relation 
\begin{equation*}
  \tilde{F}'(z,w) = \sum_{n \ge 1} \sum_{r \ge 0} \mathbb{E}(D_{n}^{\underline{r}}) z^{n-1} \frac{w^{r}}{r!}, 
\end{equation*}
with $\mathbb{E}(D_{n}^{\underline{r}}) = \mathbb{E}(D_{n} \cdot (D_{n}-1) \cdots (D_{n}-r+1))$ the $r$-th factorial moment of $D_{n}$. Plugging $v=1+w$ into \eqref{eqn:gf_sourcedeg_explicit}, extracting coefficients and applying Stirling's formula for the factorials easily gives the following explicit and asymptotic result for the $r$-th factorial moments of $D_{n}$, with $r \ge 1$:
\begin{equation*}
  \mathbb{E}(D_{n}^{\underline{r}}) = r! \sum_{j=0}^{r-1} \binom{r-1}{j} (-1)^{r-1-j} \binom{n+p(j+1)-1}{n-1} \sim 
	\frac{r! \cdot n^{rp}}{\Gamma(rp+1)}.
\end{equation*}
Due to $\mathbb{E}(D_{n}^{r}) \sim \mathbb{E}(D_{n}^{\underline{r}})$, for $r$ fixed and $n \to \infty$, we further deduce
\begin{equation}\label{eqn:Bernoulli_Dn_moments_asy}
  \mathbb{E}\left(\Big(\frac{D_{n}}{n^{p}}\Big)^{r}\right) \sim \frac{r!}{\Gamma(rp+1)}, \quad \text{for $r \ge 1$}.
\end{equation}
Thus, according to \eqref{eqn:Bernoulli_Dn_moments_asy}, the integer moments of the suitably scaled r.v.\ $D_{n}$ converge to the integer moments of a Mittag-Leffler distributed r.v.\ with parameter $p$, which, by an application of the theorem of Fr\'echet and Shohat (see, e.g., \cite{Loe1977}), indeed characterizes the limiting distribution of $D_{n}$ as stated.

We further remark that by starting with the explicit formula \eqref{eqn:gf_sourcedeg_explicit} it is also possible to characterize the limiting variable $D$ via its density function $f(x)$ (and thus to establish a local limit theorem); we only give a raw sketch. Namely, it holds
\begin{equation}\label{eqn:probability_Hankel}
  \mathbb{P}\{D_{n}=m\} = [z^{n-1}v^{m}] F'(z,v) = \frac{1}{2 \pi i} \oint \frac{(1-(1-z)^{p})^{m-1}}{z^{n} (1-z)^{1-p}} dz,
\end{equation}
where we have to choose as contour a positively oriented simple closed curve around the origin, which lies in the domain of analyticity of the integrand. To evaluate the integral asymptotically (and uniformly), for $m = O(n^{p+\delta})$, $\delta > 0$ and $n \to \infty$, one can adapt the considerations done in \cite{Panholzer2004} for the particular instance $p=\frac{1}{2}$. After straightforward computations, where the main contribution of the integral is obtained after substituting $z=1+\frac{t}{n}$ and exponential approximations of the integrand, one gets the following asymptotic equivalent of these probabilities, which determines the density function $f(x)$ of the limiting distribution (with $\mathcal{H}$ a Hankel contour):
\begin{equation*}
  \mathbb{P}\{D_{n} = m\} \sim \frac{1}{n^{p}} \cdot \frac{1}{2\pi i} \int_{\mathcal{H}} \frac{e^{-t-\frac{m}{n^{p}} (-t)^{p}}}{(-t)^{1-p}} dt
	= \frac{1}{n^{p}} \cdot f(\frac{m}{n^{p}}).
\end{equation*}
\end{proof}

\subsection{Length of a random path from source to sink\label{ssec:Bernoulli_Pathlength}}

We consider the length $L_{n} = L_{n}(p)$ (measured by the number of edges) of a random path from the source to the sink in a randomly chosen series-parallel network of size $n$ for the Bernoulli model. In this context, the following definition of a random source-to-sink path seems natural: we start at the source and walk along outgoing edges, such that whenever we reach a node of out-degree $d$, $d \ge 1$, we choose one of these outgoing edges uniformly at random, until we arrive at the sink.

We divide the study of this r.v.\ into two parts. First we consider the r.v.\ $L_{n}^{[L]}$ measuring the length of the leftmost source-to-sink path in a random series-parallel network of size $n$; the meaning of the leftmost path is, that whenever we reach a node of out-degree $d$, we choose the first (i.e., leftmost) outgoing edge. Using the representation via the recursive tree model we can reduce the distributional study of $L_{n}^{[L]}$ to the analysis of $D_{n}$ already given in Section~\ref{ssec:Bernoulli_Degree}. Second we show that for a random series-parallel network of size $n$ under the Bernoulli model $L_{n}$ and $L_{n}^{[L]}$ have the same distribution. Unfortunately, we do not see a simple symmetry argument to show this fact (such an argument easily shows that the rightmost path has the same distribution as the leftmost path, but it does not seem to explain the general situation). However, we are able to prove this in a somehow indirect manner by establishing a more involved distributional recurrence for $L_{n}$ and showing that the explicit solution for the probability distribution of $L_{n}^{[L]}$ is indeed the solution of the recurrence for $L_{n}$. 

\begin{proposition}\label{prop:Bernoulli_leftpath_limit}
The length $L_{n}^{[L]}$ of the leftmost path from source to sink in a randomly chosen series-parallel network of size $n$ generated by the Bernoulli model has the following probability distribution:
\begin{equation}\label{eqn:exact_dist_LnL}
  \mathbb{P}\{L_{n}^{[L]}=m\} = \sum_{j=0}^{m-1} \binom{m-1}{j} (-1)^{n+j-1} \binom{j-p(j+1)}{n-1}, \quad \text{for $1 \le m \le n$}.
\end{equation}

Moreover, $L_{n}^{[L]}$ converges after scaling, for $n \to \infty$, in distribution to a Mittag-Leffler distributed r.v.\ with parameter $1-p$:
\begin{equation*}
  \frac{L_{n}^{[L]}}{n^{1-p}} \xrightarrow{(d)} L = L(p), \quad \text{with} \quad L \stackrel{(d)}{=} \emph{\text{Mittag-Leffler}}(1-p).
\end{equation*}
\end{proposition}
\begin{proof}
We use that the length of the leftmost source-to-sink path in a series-parallel network has the following simple description in the corresponding edge-coloured recursive tree: namely, an edge is lying on the leftmost source-to-sink path if and only if the corresponding node in the recursive tree can be reached from the root by using only red edges (i.e., edges that correspond to serial edges). This means that the length $\ell$ of the leftmost source-to-sink path corresponds in the edge-coloured recursive tree model to the order of the maximal subtree containing the root node and only red edges. If we switch the colours red and blue in the tree we obtain an edge-coloured recursive tree where the maximal blue subtree has the same order, i.e., where the source-degree of the corresponding series-parallel network is $\ell$. But switching colours in the tree model corresponds to switching the probabilities $p$ and $q=1-p$ for generating a parallel and a serial edge, respectively, in the series-parallel network. Thus it simply holds $L_{n}^{[L]}(p) \stackrel{(d)}{=} D_{n}(1-p)$, where $D_{n}$ denotes the source-degree in a random series-parallel network of size $n$, and the stated results follow from Theorem~\ref{thm:Bernoulli_degree_limit}.
\end{proof}

\begin{theorem}\label{thm:Bernoulli_length_path}
The length $L_{n}$ of a random path and the length $L_{n}^{[L]}$ of the leftmost path from source to sink in a randomly chosen series-parallel network of size $n$ are equidistributed, $L_{n} \stackrel{(d)}{=} L_{n}^{[L]}$, thus the results of Proposition~\ref{prop:Bernoulli_leftpath_limit} are also valid for $L_{n}$.

Furthermore, the joint distribution of $L_{n}$ and of the source-degree $D_{n}$ is given as follows (with $1 \le \ell, m \le n$):
\begin{equation*}
  \mathbb{P}\{L_{n}=m \; \text{and} \; D_{n}=\ell\} = \sum_{i=0}^{m-1} \sum_{j=0}^{\ell-1} \binom{m-1}{i} \binom{\ell-1}{j} (-1)^{n+i+j-1} \binom{(1-p)i + pj}{n-1}.
\end{equation*}
\end{theorem}
\begin{proof}
In order to treat $L_{n}$ we find it necessary to give a joint study of $(L_{n}, D_{n})$, with $D_{n}$ the source-degree analysed in Section~\ref{ssec:Bernoulli_Degree}. In the following we use abbreviations for the corresponding probability mass functions:
\begin{equation*}
\begin{split}
  P_{n,m,\ell} & := \mathbb{P}\{L_{n}=m \; \text{and} \; D_{n}=\ell\}, \qquad P_{n,m} := \mathbb{P}\{L_{n}=m\} = \sum_{\ell=1}^{n} P_{n,m,\ell},\\
	A_{n,\ell} & := \mathbb{P}\{D_{n}=\ell\} = \sum_{m=1}^{n} P_{n,m,\ell}.
\end{split}
\end{equation*}
Furthermore we use the abbreviation $G = \text{Net}(T)$ to indicate that $G$ is the series-parallel network corresponding to an edge-coloured recursive tree $T$.
To establish a recurrence for $P_{n,m,\ell}$ we again use the description of the growth of the networks via edge-coloured recursive trees. In contrast to the study of $D_{n}$ given in the previous section here it seems advantageous to use an alternative decomposition of recursive trees with respect to the edge connecting nodes $1$ and $2$. Namely, it is not difficult to show (see, e.g., \cite{DobFil1999}) that when starting with a random recursive tree $T$ of order $n \ge 2$ and removing the edge $1\!\!-\!\!2$, both resulting trees $T'$ and $T''$ are (after an order-preserving relabelling) again random recursive trees of smaller orders; moreover, if $U_{n}$ denotes the order of the resulting tree $T'$ rooted at the former label $2$ (and thus $n-U_{n}$ gives the order of the tree $T''$ rooted at the original root of the tree $T$), it holds (see, e.g., \cite{vHofHoovMie2002}) that $U_{n}$ follows a discrete uniform distribution on the integers $\{1, \dots, n-1\}$, i.e., $\mathbb{P}\{U_{n} = k\} = \frac{1}{n-1}$, for $1 \le k \le n-1$. Depending on the colour of the edge $1\!\!-\!\!2$ in the edge-labelled recursive tree, it corresponds to a parallel edge (colour blue, which occurs with probability $p$) or a serial edge (colour red, which occurs with probability $q=1-p$) in the series-parallel network. If it is a serial edge then the length of a random path in $\text{Net}(T)$ is the sum of the lengths of random paths in $\text{Net}(T')$ and $\text{Net}(T'')$; furthermore the source-degree of $\text{Net}(T)$ corresponds to the source-degree of $\text{Net}(T'')$. On the other hand if $1\!\!-\!\!2$ is a parallel edge then the source-degree $d$ of $\text{Net}(T)$ is the sum of the respective source-degrees $d'$ and $d''$ of $\text{Net}(T')$ and $\text{Net}(T'')$, whereas the length of a random path in $\text{Net}(T)$ is with probability $\frac{d'}{d}$ the length of a random path in $\text{Net}(T')$ and with probability $\frac{d''}{d}$ the length of a random path in $\text{Net}(T'')$. 

These considerations yield the following recurrence for $P_{n,m,\ell}$, for $1 \le \ell,m \le n$, where outside this range we assume that $P_{n,m,\ell}=0$:
\begin{align}
  P_{n,m,\ell} & = \frac{1-p}{n-1} \sum_{k=1}^{n-1} \sum_{i=1}^{m-1} P_{k,i} P_{n-k,m-i,\ell}
	+ \frac{p}{n-1} \sum_{k=1}^{n-1} \sum_{j=1}^{\ell-1} \Big(\frac{j}{\ell} P_{k,m,j} A_{n-k,\ell-j} + \frac{\ell-j}{\ell} P_{n-k,m,\ell-j} A_{k,j}\Big),\notag\\
	& \quad \text{for $n \ge 2, 1 \le \ell,m \le n$},\label{eqn:Bernoulli_pathlength_recurrence}\\
	P_{1,1,1,} & = 1.\notag
\end{align}
In order to treat recurrence~\eqref{eqn:Bernoulli_pathlength_recurrence} we introduce the generating function 
\begin{equation*}
  G(z,v,w) := \sum_{n \ge 1} \sum_{m \ge 1} \sum_{\ell \ge 1} P_{n,m,\ell} z^{n-1} v^{m} w^{\ell}, 
\end{equation*}
which thus satisfies $G(z,v,1) = \sum\limits_{n \ge 1} \sum\limits_{m \ge 1} P_{n,m} z^{n-1} v^{m}$ and $G(z,1,w) = \sum\limits_{n \ge 1} \sum\limits_{\ell \ge 1} A_{n,\ell} z^{n-1} w^{\ell}$.\newline
Straightforward computations lead to the following functional-differential equation for $G(z,v,w)$:
\begin{equation}\label{eqn:Bernoulli_pathlength_Gzvw_DEQ}
  \frac{\partial^{2}}{\partial{w} \partial{z}}G(z,v,w) = (1-p) G(z,v,1) \frac{\partial}{\partial w}G(z,v,w) + 2p G(z,1,w) \frac{\partial}{\partial w}G(z,v,w),
\end{equation}
with side conditions $G(0,v,w) = vw$, $G(z,v,0)=0$ and $G(z,0,w)=0$.

Although it is not apparent how to solve such an equation we can guess the solution of \eqref{eqn:Bernoulli_pathlength_Gzvw_DEQ}: namely, it is not difficult to give a joint study of $(L_{n}^{[L]}, D_{n})$ in the recursive tree model, where it corresponds to a joint study of the order of the red subtree and of the blue subtree, by extending the approach given in Section~\ref{ssec:Bernoulli_Degree}. This yields the corresponding generating function
\begin{equation}\label{eqn:Bernoulli_pathlength_Gzvw_sol}
  G(z,v,w) = \frac{vw}{\big(1-v(1-(1-z)^{1-p})\big) \cdot \big(1-w(1-(1-z)^{p})\big)}.
\end{equation}
However, it is an easy task to verify (by differentiating and evaluating) that $G(z,v,w)$ given by \eqref{eqn:Bernoulli_pathlength_Gzvw_sol} is indeed the solution of \eqref{eqn:Bernoulli_pathlength_Gzvw_DEQ} (we omit these straightforward computations). Thus it even holds $(L_{n}, D_{n}) \stackrel{(d)}{=} (L_{n}^{[L]}, D_{n})$, which of course implies the corresponding statement of the theorem. Moreover, extracting coefficients from \eqref{eqn:Bernoulli_pathlength_Gzvw_sol} according to $\mathbb{P}\{L_{n}=m \; \text{and} \; D_{n}=\ell\} = P_{n,m,\ell} = [z^{n-1} v^{m} w^{\ell}] G(z,v,w)$ characterizes the joint distribution of $L_{n}$ and $D_{n}$.
\end{proof}

\subsection{Number of paths from source to sink}

Let $P_{n} = P_{n}(p)$ denote the r.v.\ measuring the number of different paths from the source to the sink in a randomly chosen series-parallel network of size $n$ for the Bernoulli model. We obtain the following theorem for the expected number of source-to-sink paths.
\begin{theorem}
The expectation $\mathbb{E}(P_{n})$ of the number of paths $P_{n}$ from source to sink in a random series-parallel network of size $n$ generated by the Bernoulli model is given by the following explicit formula:
\begin{equation*}
  \mathbb{E}(P_{n}) = 
	\begin{cases}
	  \sum\limits_{j=0}^{n-1} (-1)^{n+j-1} \binom{(2p-1)j-1}{n-1} \sum\limits_{k=0}^{n-1} \binom{k}{j} \left(\frac{p}{2p-1}\right)^{k}, & \quad \text{for $p \neq \frac{1}{2}$},\\
		\sum\limits_{k=0}^{n-1} \frac{(-1)^{k}}{2^{k}} \cdot B_{k}\big(-H_{n-1}^{(1)},-H_{n-1}^{(2)},-2H_{n-1}^{(3)},\ldots,-(k-1)! H_{n-1}^{(k)}\big), & \quad \text{for $p=\frac{1}{2}$},
	\end{cases}
\end{equation*}
where $B_{k}(x_{1}, x_{2}, \ldots, x_{k})$ denotes the $k$-th complete Bell polynomial and where $H_{n}^{(m)} := \sum_{j=1}^{n} \frac{1}{j^{m}}$ denote the $m$-th order harmonic numbers.

The asymptotic behaviour of $\mathbb{E}(P_{n})$ is, for $n \to \infty$, given as follows:
\begin{gather*}
  \mathbb{E}(P_{n}) = \frac{1}{1-p} \cdot \alpha_{p}^{n} + R_{p}(n),\\
\begin{split}
& \text{where $\alpha_{p} = \frac{1}{1-\big(\frac{p}{1-p}\big)^{\frac{1}{1-2p}}}$, for $p \neq \frac{1}{2}$, \enspace and \enspace
$\alpha_{p} = \frac{1}{1-e^{-2}} = \lim_{p \to \frac{1}{2}}\frac{1}{1-\big(\frac{p}{1-p}\big)^{\frac{1}{1-2p}}}$, for $p = \frac{1}{2}$},
\end{split}\\
\text{and 
$R_{p}(n) = 
\begin{cases}
  - \frac{1-2p}{p \Gamma(2p)} n^{2p-1} + \mathcal{O}(n^{2(2p-1)}), & \quad \text{for $0<p<\frac{1}{2}$},\\
	-\frac{2}{\log n} + \mathcal{O}(\frac{1}{\log^{2}n}), & \quad \text{for $p=\frac{1}{2}$},\\
	- \frac{2p-1}{1-p} + \mathcal{O}(n^{1-2p}), & \quad \text{for $\frac{1}{2}<p<1$}.
\end{cases}
$}
\end{gather*}
\end{theorem}
\begin{proof}
We use the description of the growth of the networks via edge-coloured recursive trees, where we use the decomposition of recursive trees with respect to the edge $1\!\!-\!\!2$ as described in the proof of Theorem~\ref{thm:Bernoulli_length_path}. If this edge is coloured blue (thus corresponding to a parallel doubling in the network) then the number of source-to-sink paths in the corresponding substructures have to be added, whereas if it is coloured red (i.e., corresponding to a serial doubling) they have to be multiplied in order to obtain the total number of source-to-sink paths in the whole graph. Thus $P_{n}$ satisfies the following stochastic recurrence:
\begin{equation}\label{eqn:nrpath_stochrec}
  P_{n} \stackrel{(d)}{=} \boldsymbol{1}_{\{B_{n}=1\}} \cdot \left(P_{U_{n}}' + P_{n-U_{n}}''\right) + \boldsymbol{1}_{\{B_{n}=0\}} \cdot \left(P_{U_{n}}' \cdot P_{n-U_{n}}''\right), \quad \text{for $n \ge 2$}, \quad P_{1} = 1,
\end{equation}
where $B_{n}$ and $U_{n}$ are independent of each other and independent of $(P_{k})_{k \ge 1}$, $(P_{k}')_{k \ge 1}$ and $(P_{k}'')_{k \ge 1}$, and where $P_{k}'$ and $P_{k}''$ are independent copies of $P_{k}$, for $k \ge 1$. Here $B_{n}$ is the indicator variable of the event that $1\!\!-\!\!2$ is a blue edge in the recursive tree, thus $B_{n}$ is a Bernoulli distributed r.v.\ with success probability $p$, i.e., $\mathbb{P}\{B_{n} = 1\} = p$. Furthermore, the r.v.\ $U_{n}$ measuring the order of the subtree rooted at $2$, is uniformly distributed on $\{1, 2, \dots, n-1\}$, i.e., $\mathbb{P}\{U_{n} = k\} = \frac{1}{n-1}$, for $1 \le k \le n-1$.

Starting with \eqref{eqn:nrpath_stochrec} and taking the expectations yields after simple manipulations the following recurrence:
\begin{equation}\label{eqn:Bernoulli_paths_Pn_rec}
\begin{split}
  \mathbb{E}(P_{n}) & = \frac{2p}{n-1} \sum_{k=1}^{n-1} \mathbb{E}(P_{k}) + \frac{1-p}{n-1} \sum_{k=1}^{n-1} \mathbb{E}(P_{k}) \mathbb{E}(P_{n-k}), \quad n \ge 2, \qquad \mathbb{E}(P_{1})=1.
\end{split}
\end{equation}
To treat recurrence~\eqref{eqn:Bernoulli_paths_Pn_rec} we introduce the generating function $E(z) := \sum_{n \ge 1} \mathbb{E}(P_{n}) z^{n-1}$, which gives the following first order non-linear differential equation of Bernoulli type:
\begin{equation}\label{eqn:DEQ_Bernoulli_nrpath}
  E'(z) = \frac{2p}{1-z} E(z) + (1-p) \big(E(z)\big)^{2}, \quad E(0)=1.
\end{equation}

Equation~\eqref{eqn:DEQ_Bernoulli_nrpath} can be treated by a standard technique for Bernoulli type differential equations and leads to the following solution, where we have to distinguish whether $p=\frac{1}{2}$ or not:
\begin{equation}\label{eqn:GF_exp_Ez}
  E(z) = 
  \begin{cases}
    \frac{1-2p}{(1-p)(1-z) - p(1-z)^{2p}}, & \quad \text{for $p \neq \frac{1}{2}$},\\
    \frac{2}{2(1-z) - (1-z)\log\left(\frac{1}{1-z}\right)}, & \quad \text{for $p = \frac{1}{2}$}.
	\end{cases}
\end{equation}

From formula \eqref{eqn:GF_exp_Ez} for the generating function $E(z)$ one can deduce the explicit results for the expected value $\mathbb{E}(P_{n}) = [z^{n-1}] E(z)$ stated in the theorem. Whereas for $p \neq \frac{1}{2}$ extracting coefficients is completely standard, for $p=\frac{1}{2}$ we use the description of the coefficients of the functions $\big(\log \frac{1}{1-z}\big)^{k}$ via Bell polynomials and higher order harmonic numbers given in \cite{Zave1976}. However, due to alternating signs of the summands these explicit formul{\ae} are not easily amenable for asymptotic considerations. Instead, in order to obtain the asymptotic behaviour of $\mathbb{E}(P_{n})$ we consider the formul{\ae} for the generating function $E(z)$ stated in \eqref{eqn:GF_exp_Ez} and describe the structure of the singularities: for $0 < p < 1$ the dominant singularity at $z=\rho<1$ is annihilating the denominator; there $E(z)$ has a simple pole, which due to singularity analysis \cite{FlaSed2009} yields the main term of $\mathbb{E}(P_{n})$, i.e., the asymptotically exponential growth behaviour; the (algebraic or logarithmic) singularity at $z=1$ determines the second and higher order terms in the asymptotic behaviour of $\mathbb{E}(P_{n})$, which differ according to the ranges $0<p<\frac{1}{2}$, $p=\frac{1}{2}$, and $\frac{1}{2} < p < 1$. The theorem stated for the asymptotic behaviour of $\mathbb{E}(P_{n})$ is an immediate consequence of the following singular expansion of $E(z)$, which can be obtained in a straightforward way by carrying out above considerations; here the dominant singularity $\rho$ is given by $\rho := \rho_{p} = \frac{1}{\alpha_{p}}$, with $\alpha_{p}$ stated in the theorem:
\begin{equation*}
  E(z) =
	\begin{cases}
	  \frac{1}{(1-p) \rho \big(1-\frac{z}{\rho}\big)} - \frac{1-2p}{p (1-z)^{2p}} + \mathcal{O}\big((1-z)^{1-4p}\big), & \quad \text{for $0 < p < \frac{1}{2}$},\\
		\frac{2}{\rho \big(1-\frac{z}{\rho}\big)} - \frac{2}{(1-z)\log\big(\frac{1}{1-z}\big)} + \mathcal{O}\big(\frac{1}{(1-z) \log^{2}\big(\frac{1}{1-z}\big)}\big), & \quad \text{for $p=\frac{1}{2}$},\\
	  \frac{1}{(1-p) \rho \big(1-\frac{z}{\rho}\big)} - \frac{2p-1}{(1-p)(1-z)} + \mathcal{O}\big((1-z)^{2p-2}\big), & \quad \text{for $\frac{1}{2} < p < 1$}.
	\end{cases}
\end{equation*}
\end{proof}

\section{Uniform binary saturation edge-duplication growth model\label{sec:Binary_model}}

\subsection{Length of a random path from source to sink}

We are interested in the length of a typical source-to-sink path in a series-parallel network of size $n$. Again, it is natural to start at the source of the graph and move along outgoing edges, in a way that whenever we have the choice of two outgoing edges we use one of them uniformly at random to enter a new node, until we finally end at the sink. Let us denote by $L_{n}$ the length of such a random source-to-sink path in a random series-parallel network of size $n$ for the binary model. We collect our findings for the r.v.\ $L_{n}$ in the next theorem, where we also restate the result for the expectation $\mathbb{E}(L_{n})$ obtained in \cite{Mahmoud2014}.
\begin{theorem}
  Let $L_{n}$ be the r.v.\ measuring the length of a random path from source to sink in a random series-parallel network of size $n$ generated by the binary model. The expectation of $L_{n}$ is given by the following exact and asymptotic formul{\ae}:
	\begin{equation*}
  \mathbb{E}(L_{n}) = n \left(\frac{3+\sqrt{5}}{2 \sqrt{5}} \binom{n+\frac{\sqrt{5}}{2}-\frac{3}{2}}{n} - \frac{3-\sqrt{5}}{2 \sqrt{5}} \binom{n-\frac{\sqrt{5}}{2} - \frac{3}{2}}{n}\right) \sim \frac{1+\sqrt{5}}{2 \sqrt{5}} \frac{n^{\frac{\sqrt{5}-1}{2}}}{\Gamma(\frac{\sqrt{5}-1}{2})}.
	\end{equation*}
  $L_{n}$ satisfies, for $n \to \infty$, the following limiting distribution behaviour (with $\phi=\frac{\sqrt{5}-1}{2}$):
	\begin{equation*}
	  \frac{L_{n}}{n^{\phi}} \xrightarrow{(d)} L=L(p),
	\end{equation*}
	where the limit $L$ is characterized by its sequence of $r$-th integer moments via
	\begin{equation*}
	  \mathbb{E}(L^{r}) = \frac{r! \cdot c_{r}}{\Gamma(r \phi +1)}, \quad \text{for $r \ge 0$},
	\end{equation*}
	where the sequence $c_{r}$ satisfies the recurrence (with $c_{0}=1$ and $c_{1} = \frac{3+\phi}{5}$):
	\begin{equation*}
	  c_{r} = \frac{1}{(r-1) \phi ((r+1)\phi+1)} \sum_{k=1}^{r-1} (k\phi+1) c_{k} c_{r-k}, \quad \text{for $r \ge 2$}.
	\end{equation*}
\end{theorem}

\begin{proof}
Due to symmetry reasons it holds that $L_{n} \stackrel{(d)}{=} L_{n}^{[L]}$, where $L_{n}^{[L]}$ denotes the length of the leftmost source-to-sink path in a random series-parallel network of size $n$, i.e., the source-to-sink path, where in each node we choose the left outgoing edge to enter the next node.

In order to analyse $L_{n}^{[L]}$ we use the description of the growth of series-parallel networks via bucket recursive trees: the length of the left path is equal to $1$ (coming from the root node of the tree, i.e., stemming from the edge $1$ in the graph) plus the sum of the lengths of the left paths in the subtrees contained in the left forest (which correspond to the blocks of the left part of the graph). When we introduce the generating function
\begin{equation*}
  F(z,v) := \sum_{n \ge 1} \sum_{m \ge 0} T_{n} \mathbb{P}\{L_{n}=m\} \frac{z^{n}}{n!} v^{m}
	= \sum_{n \ge 1} \sum_{m \ge 0} \mathbb{P}\{L_{n}=m\} \frac{z^{n}}{n} v^{m},
\end{equation*}
then above description yields the following differential equation:
\begin{equation}\label{eqn_DEQ_binary_length}
  F''(z,v) = v e^{F(z,v)} e^{N(z)} = \frac{v}{1-z} e^{F(z,v)}, \quad F(0,v)=0, \quad F'(0,v)=v,
\end{equation}
where $N(z) := \sum_{n \ge 1} T_{n} \frac{z^{n}}{n!} = \log\big(\frac{1}{1-z}\big)$ is the exponential generating function of the number $T_{n} = (n-1)!$ of bucket recursive trees of order $n$. In order to compute the expectation we consider $E(z) := \left.\frac{\partial}{\partial v} F(z,v)\right|_{v=1} = \sum_{n \ge 1} \mathbb{E}(L_{n}) \frac{z^{n}}{n}$, which satisfies the following second order linear differential equation of Eulerian type:
\begin{equation*}
  E''(z) = \frac{1}{(1-z)^{2}} E(z) + \frac{1}{(1-z)^{2}}, \quad E(0)=0, \quad E'(0)=1.
\end{equation*}
The explicit solution of this equation can be obtained by standard techniques and is given as follows:
\begin{equation}\label{eqn_binary_length_Ez_exp}
  E(z) = \frac{3+\sqrt{5}}{2 \sqrt{5}} \frac{1}{(1-z)^{\frac{\sqrt{5}-1}{2}}} - \frac{3-\sqrt{5}}{2 \sqrt{5}} (1-z)^{\frac{1+\sqrt{5}}{2}} - 1.
\end{equation}
Extracting coefficients from it and applying Stirling's formula immediately yields the explicit and asymptotic result for the expectation obtained by Mahmoud in \cite{Mahmoud2014} and that is restated in the theorem.

In order to characterize the limiting distribution of $L_{n}$ we will compute iteratively the asymptotic behaviour of all its integer moments. To this aim it is advantageous to consider $G(z,v) := F'(z,v)$. Differentiating \eqref{eqn_DEQ_binary_length} shows that $G(z,v)$ satisfies the following differential equation:
\begin{equation}\label{eqn:DEQ_binary_length_G}
  G''(z,v) = G'(z,v) G(z,v) + \frac{1}{1-z} G'(z,v), \quad G(0,v)=v, \quad G'(0,v) = v.
\end{equation}
We introduce the generating functions $M_{r}(z) := \left.\frac{\partial^{r}}{\partial v^{r}} G(z,v)\right|_{v=1} = \sum_{n \ge 1} \mathbb{E}(L_{n}^{\underline{r}}) z^{n-1}$ of the $r$-th factorial moments of $D_{n}$. According to the definition it holds $M_{0}(z) = \frac{1}{1-z}$, whereas $M_{1}(z) = E'(z)$, with $E(z)$ given in \eqref{eqn_binary_length_Ez_exp}.

For $r \ge 2$, differentiating \eqref{eqn:DEQ_binary_length_G} $r$ times w.r.t.\ $v$ and evaluating at $v=1$ yields
\begin{equation}\label{eqn:binary_length_Mrz_DEQ}
\begin{split}
  M_{r}''(z) & = \frac{2}{1-z} M_{r}'(z) + \frac{1}{(1-z)^{2}} M_{r}(z) + R_{r}(z), \quad M_{r}(0) = M_{r}'(0) = 0,\\
	& \quad \text{with} \quad R_{r}(z) = \sum_{k=1}^{r-1} \binom{r}{k} M_{k}'(z) M_{r-k}(z).
\end{split}
\end{equation}
Thus $M_{r}(z)$ satisfies for each $r \ge 2$ an inhomogeneous Eulerian differential equation, where the inhomogeneous part $R_{r}(z)$ depends on the functions $M_{k}(z)$, with $k < r$. The solution of \eqref{eqn:binary_length_Mrz_DEQ} satisfying the given initial conditions can be obtained by standard techniques and is given as follows:
\begin{equation}\label{eqn:binary_length_Mrz_sol}
  M_{r}(z) = \frac{1}{\sqrt{5} \cdot (1-z)^{\frac{1+\sqrt{5}}{2}}} \cdot \int_{0}^{z} (1-t)^{\frac{3+\sqrt{5}}{2}} R_{r}(t) dt
	- \frac{(1-z)^{\frac{\sqrt{5}-1}{2}}}{\sqrt{5}} \cdot \int_{0}^{z} (1-t)^{\frac{3-\sqrt{5}}{2}} R_{r}(t) dt.
\end{equation}
From the representation \eqref{eqn:binary_length_Mrz_sol} it immediately follows by induction that $z=1$ is the unique dominant singularity of the functions $M_{r}(z)$. Furthermore, it can be shown inductively that the local behaviour of $M_{r}(z)$ in a complex neighbourhood of $z=1$ is given by
\begin{equation}\label{eqn:binary_length_Mrz_asy}
  M_{r}(z) \sim \frac{\tilde{c}_{r}}{(1-z)^{r \phi+1}}, \quad \text{for $r \ge 0$},
\end{equation}
with $\phi = \frac{\sqrt{5}-1}{2}$ and certain constants $\tilde{c}_{r}$. Namely, from the explicit results for $M_{0}(z)$ and $M_{1}(z)$ we obtain $\tilde{c}_{0} = 1$ and $\tilde{c}_{1} = \frac{1+\sqrt{5}}{2\sqrt{5}} = \frac{3+\phi}{5}$, whereas \eqref{eqn:binary_length_Mrz_sol} yields by applying the induction hypothesis and closure properties of singular integration and differentiation (see \cite{FlaSed2009}) the local expansion
\begin{equation*}
  M_{r}(z) \sim \frac{1}{(1-z)^{r \phi +1}} \cdot \frac{1}{\sqrt{5}} \left(\frac{1}{(r-1)\phi} - \frac{1}{(r+1)\phi+1}\right) \sum_{k=1}^{r-1} \binom{r}{k} (k \phi+1) \tilde{c}_{k} \tilde{c}_{r-k},
\end{equation*}
which, after simple manipulations, characterizes the sequence $\tilde{c}_{r}$ via the following recurrence of ``convolution type'':
\begin{equation*}
  \tilde{c}_{r} = \frac{1}{(r-1) \phi ((r+1)\phi+1)} \sum_{k=1}^{r-1} \binom{r}{k} (k \phi+1) \tilde{c}_{k} \tilde{c}_{r-k}, \quad r \ge 2.
\end{equation*}
Taking into account $\mathbb{E}(L_{n}^{r}) \sim \mathbb{E}(L_{n}^{\underline{r}}) = [z^{n-1}] M_{r}(z)$ and extracting coefficients from \eqref{eqn:binary_length_Mrz_asy} by applying basic singularity analysis yields
\begin{equation*}
  \mathbb{E}\left(\Big(\frac{L_{n}}{n^{\phi}}\Big)^{r}\right) \sim \frac{\tilde{c}_{r}}{\Gamma(r \phi + 1)}, \quad \text{for $r \ge 0$}.
\end{equation*}
Thus an application of the theorem of Fr\'{e}chet and Shohat shows the limiting distribution result stated in the theorem.
\end{proof}

\subsection{Degree of the sink}

Whereas the (out-)degree of the source of a binary series-parallel network is two (if the graph has at least two edges), typically the (in-)degree of the sink is quite large, as will follow from our treatments. Let us denote by $D_{n}$ the degree of the sink in a random series-parallel network of size $n$ for the binary model. In the following we state our results on the distributional behaviour of $D_{n}$.
\begin{theorem}
  Let $D_{n}$ be the r.v.\ measuring the degree of the sink in a random series-parallel network of size $n$ generated by the binary model. The expectation of $D_{n}$ is given by the following exact and asymptotic formul{\ae}:
	\begin{equation*}
  \mathbb{E}(D_{n}) = \frac{1+\sqrt{2}}{2} \binom{n+\sqrt{2}-2}{n-1} - \frac{\sqrt{2}-1}{2} \binom{n-\sqrt{2}-2}{n-1} \sim \frac{1+\sqrt{2}}{2} \frac{n^{\sqrt{2}-1}}{\Gamma(\sqrt{2})}.
\end{equation*}
  $D_{n}$ satisfies, for $n \to \infty$, the following limiting distribution behaviour:
	\begin{equation*}
	  \frac{D_{n}}{n^{\sqrt{2}-1}} \xrightarrow{(d)} D=D(p),
	\end{equation*}
	where the limit $D$ is characterized by its sequence of $r$-th integer moments via
	\begin{equation*}
	  \mathbb{E}(D^{r}) = \frac{r! (r(\sqrt{2}-1)+1) c_{r}}{\Gamma(r (\sqrt{2}-1) +1)}, \quad \text{for $r \ge 0$},
	\end{equation*}
	where the sequence $c_{r}$ satisfies the recurrence (with $c_{0}=1$ and $c_{1} = \frac{1+\sqrt{2}}{2 \sqrt{2}}$):
	\begin{equation*}
	  c_{r} = \frac{1}{(r(\sqrt{2}-1)+1)^{2}-2} \sum_{k=1}^{r-1} c_{k} c_{r-k}, \quad \text{for $r \ge 2$}.
	\end{equation*}
\end{theorem}

\begin{proof}
For a binary series-parallel network, the value of this parameter can be determined recursively by adding the degrees of the sinks in the last block of each half of the graph; in the case that a half only consists of one edge then the contribution of this half is of course $1$. When considering the corresponding bucket recursive tree this means that the degree of the sink can be computed recursively by adding the contributions of the left and the right forest attached to the root, where the contribution of a forest is either given by $1$ in case that the forest is empty (then the corresponding root node contributes to the degree of the sink) or it is the contribution of the first tree in the forest (which corresponds to the last block), see Figure~\ref{fig:link_decomposition_buckettree_network}.
Introducing the generating functions
\begin{equation*}
  F(z,v) := \sum_{n \ge 1} \sum_{m \ge 1} T_{n} \mathbb{P}\{D_{n}=m\} \frac{z^{n}}{n!} v^{m}, \quad
	A(z,v) := \sum_{n \ge 0} \sum_{m \ge 1} \tilde{T}_{n} \mathbb{P}\{\tilde{D}_{n}=m\} \frac{z^{n}}{n!} v^{m},
\end{equation*}
with $\tilde{D}_{n}$ denoting the corresponding quantity for the left or right forest and $\tilde{T}_{n} = n!$ counting the number of forests of order $n$, the combinatorial decomposition of bucket recursive trees yields the following system of differential equations:
\begin{equation}\label{eqn:Fzv_Azv}
  F''(z,v) = \big(A(z,v)\big)^{2}, \quad A'(z,v) = \frac{1}{1-z} \cdot F'(z,v).
\end{equation}
From system \eqref{eqn:Fzv_Azv} the following non-linear differential equation for $F(z,v)$ can be obtained:
\begin{equation*}
  F'''(z,v) = \frac{2}{1-z} \sqrt{F''(z,v)} F'(z,v), \quad F(0,v)=0, F'(0,v)=v, F''(0,v)=v^{2}.
\end{equation*}
Introducing $E(z) := \left.\frac{\partial}{\partial v}F'(z,v)\right|_{v=1} = \sum_{n \ge 1} \mathbb{E}(D_{n}) z^{n-1}$ and solving an Eulerian differential equation for it yields the explicit solution
\begin{equation}\label{eqn:degree_binary_Ez_sol}
  E(z) = \frac{1+\sqrt{2}}{2 (1-z)^{\sqrt{2}}} - \frac{\sqrt{2}-1}{2} (1-z)^{\sqrt{2}},
\end{equation}
from which the stated results for $\mathbb{E}(D_{n})$ easily follow.

However, for asymptotic studies of higher moments it seems to be advantageous to consider the following second order non-linear differential equation for $A(z,v)$, which follows immediately from \eqref{eqn:Fzv_Azv}:
\begin{equation}\label{eqn:DEQ_degree_binary_A}
  A''(z,v) = \frac{1}{1-z} A'(z,v) + \frac{1}{1-z} \big(A(z,v)\big)^{2}, \quad A(0,v)=v, \quad A'(0,v)=v.
\end{equation}
We introduce the functions $\tilde{M}_{r}(z) := \left.\frac{\partial^{r}}{\partial v^{r}}A(z,v)\right|_{v=1} = \sum_{n \ge 0} \mathbb{E}(\tilde{D}_{n}^{\underline{r}}) z^{n}$. According to the definition it holds $\tilde{M}_{0}(z) = \frac{1}{1-z}$, whereas \eqref{eqn:Fzv_Azv} yields the relation $E(z) = (1-z) \tilde{M}_{1}(z)$, with $E(z)$ given by \eqref{eqn:degree_binary_Ez_sol}, from which we obtain
\begin{equation*}
  \tilde{M}_{1}(z) = \frac{1+\sqrt{2}}{2 \sqrt{2}} \frac{1}{(1-z)^{\sqrt{2}}} + \frac{\sqrt{2}-1}{2 \sqrt{2}} (1-z)^{\sqrt{2}}.
\end{equation*}
For $r \ge 2$, differentiating \eqref{eqn:DEQ_degree_binary_A} $r$ times w.r.t.\ $v$ and evaluating at $v=1$ shows that $\tilde{M}_{r}(z)$ satisfies the following second order Eulerian differential equation:
\begin{equation}\label{eqn:DEQ_degree_binary_moments}
\begin{split}
  \tilde{M}_{r}''(z) & = \frac{1}{1-z} \tilde{M}_{r}'(z) + \frac{2}{(1-z)^{2}} \tilde{M}_{r}(z) + R_{r}(z), \quad \tilde{M}_{r}(0) \tilde{M}_{r}'(0) =0,\\
	& \quad \text{with} \quad R_{r}(z) = \sum_{k=1}^{r-1} \binom{r}{k} M_{k}'(z) M_{r-k}(z).
\end{split}
\end{equation}
Applying standard techniques give the solution of \eqref{eqn:DEQ_degree_binary_moments}:
\begin{equation}\label{eqn:binary_degree_Mrz_sol}
  \tilde{M}_{r}(z) = \frac{1}{(1-z)^{\sqrt{2}}} \frac{1}{2 \sqrt{2}} \int_{0}^{z} (1-t)^{\sqrt{2}+1} R_{r}(t) dt - (1-z)^{\sqrt{2}} \frac{1}{2 \sqrt{2}} \int_{0}^{z} \frac{1}{(1-t)^{\sqrt{2}-1}} R_{r}(t) dt.
\end{equation}
An inductive argument shows thus that $z=1$ is the unique dominant singularity of the functions $\tilde{M}_{r}(z)$. Furthermore, again via induction one can prove that the local behaviour of $\tilde{M}_{r}(z)$ in a complex neighbourhood of $z=1$ is given by
\begin{equation}\label{eqn:binary_degree_Mrz_asy}
  \tilde{M}_{r}(z) \sim \frac{\tilde{c}_{r}}{(1-z)^{r (\sqrt{2}-1)+1}}, \quad \text{for $r \ge 0$},
\end{equation}
with certain constants $\tilde{c}_{r}$.
Namely, the explicit results for $\tilde{M}_{0}(z)$ and $\tilde{M}_{1}(z)$ yield $\tilde{c}_{0}=1$ and $\tilde{c}_{1} = \frac{1+\sqrt{2}}{2\sqrt{2}}$, whereas by applying the induction hypothesis and singular integration and differentiation one obtains from \eqref{eqn:binary_degree_Mrz_sol} for $r \ge 2$ the local expansion
\begin{equation*}
  \tilde{M}_{r}(z) \sim \frac{1}{(1-z)^{r(\sqrt{2}-1)+1}} \cdot \frac{1}{2 \sqrt{2}} \left(\frac{1}{(r-1)\sqrt{2}-r+1} - \frac{1}{(r+1)\sqrt{2}-r+1}\right) \sum_{k=1}^{r-1} \binom{r}{k} \tilde{c}_{k} \tilde{c}_{r-k},
\end{equation*}
which characterizes the sequence $\tilde{c}_{r}$ via the following recurrence:
\begin{equation*}
  \tilde{c}_{r} = \frac{1}{(r(\sqrt{2}-1)+1)^{2}-2} \sum_{k=1}^{r-1} \binom{r}{k} \tilde{c}_{k} \tilde{c}_{r-k}, \quad r \ge 2.
\end{equation*}

Actually we are interested in the functions $M_{r}(z) := \left.\frac{\partial^{r}}{\partial v^{r}}F'(z,v)\right|_{v=1} = \sum_{n \ge 1} \mathbb{E}(D_{n}^{\underline{r}}) z^{n-1}$, which are, due to \eqref{eqn:Fzv_Azv}, related to $\tilde{M}_{r}(z)$ via $M_{r}(z) = (1-z) \tilde{M}_{r}'(z)$. Thus, we get from \eqref{eqn:binary_degree_Mrz_asy}
\begin{equation*}
  M_{r}(z) \sim \frac{(r(\sqrt{2}-1)+1) \tilde{c}_{r}}{(1-z)^{r(\sqrt{2}-1)+1}}, \quad r \ge 0,
\end{equation*}
and after applying basic singularity analysis the asymptotic behaviour of the $r$-th integer moments of $D_{n}$:
\begin{equation*}
  \mathbb{E}\left(\Big(\frac{D_{n}}{n^{\sqrt{2}-1}}\Big)^{r}\right) \sim \frac{(r(\sqrt{2}-1)+1) \tilde{c}_{r}}{\Gamma(r (\sqrt{2}-1) + 1)}, \quad \text{for $r \ge 0$}.
\end{equation*}
Applying the theorem of Fr\'{e}chet and Shohat shows the stated limiting distribution result.
\end{proof}

\subsection{Number of paths from source to sink}

As for the Bernoulli model we are interested in results concerning the number of different paths from the source to the sink in a series-parallel network and denote by $P_{n}$ the number of source-to-sink paths in a random series-parallel network of size $n$ for the binary model. We obtain the following result for $\mathbb{E}(P_{n})$.
\begin{theorem}\label{thm:Binary_ExpNumberPaths}
The expectation $\mathbb{E}(P_{n})$ of the number $P_{n}$ of paths from source to sink in a random series-parallel network of size $n$ generated by the binary model has, for $n \to \infty$, the following asymptotic behaviour, with $\rho \approx 0.89\dots$:
\begin{equation*}
  \mathbb{E}(P_{n}) = \frac{2}{\rho^{n}} \cdot \left(1-\frac{\rho^{2}}{(\rho-1)^{2} (n-1) (n-2)} + \mathcal{O}\Big(\frac{\log n}{n^{4}}\Big)\right).
\end{equation*}
\end{theorem}

\begin{proof}
In order to study $P_{n}$ it seems advantageous to start with a stochastic recurrence for this random variable obtained by decomposing the bucket recursive tree into the root node and the left and right forest (of bucket recursive trees) attached to the root node. As auxiliary r.v.\ we introduce $Q_{n}$, which denotes the number of source-to-sink paths in the series-parallel network corresponding to a forest (i.e., a set) of bucket recursive trees, where each tree in the forest corresponds to a subblock in the left or right half of the graph. By decomposing the forest into its leftmost tree and the remaining set of trees and taking into account that the number of source-to-sink paths in the forest is the product of the number of source-to-sink paths in the leftmost tree and the corresponding paths in the remaining forest, we obtain the following system of stochastic recurrences:
\begin{equation}\label{eqn:number-path_binary_stochastic}
  P_{n} \stackrel{(d)}{=} Q_{U_{n}}' + Q_{n-2-U_{n}}'', \quad \text{for $n \ge 2$}, \qquad Q_{n} \stackrel{(d)}{=} P_{V_{n}}' \cdot Q_{n-V_{n}}''', \quad \text{for $n \ge 1$}, 
\end{equation}
with $P_{0}=0$, $P_{1}=1$, $Q_{0}=1$, and where the $U_{n}$, $V_{n}$ and $(P_{k}, P_{k}', Q_{k}, Q_{k}', Q_{k}'')_{k \ge 1}$ are independent. Furthermore, they are distributed as follows:
\begin{equation*}
  \mathbb{P}\{U_{n}=k\} = \frac{1}{n-1}, \quad 0 \le k \le n-2, \qquad \mathbb{P}\{V_{n}=k\} = \frac{1}{n}, \quad 1 \le k \le n.
\end{equation*}
Introducing $E_{n} := \mathbb{E}(P_{n})$ and $\tilde{E}_{n} := \mathbb{E}(Q_{n})$, the stochastic recurrence~\eqref{eqn:number-path_binary_stochastic} yields the following system of equations for $E_{n}$ and $\tilde{E}_{n}$ (with $E_{0}=0$, $E_{1}=1$ and $\tilde{E}_{0}=1$):
\begin{equation*}
  E_{n} = \frac{2}{n-1} \sum_{k=0}^{n-2} \tilde{E}_{k}, \quad n \ge 2, \qquad \tilde{E}_{n} = \frac{1}{n} \sum_{k=1}^{n} E_{k} \tilde{E}_{n-k}, \quad n \ge 1.
\end{equation*}
Introducing $E(z) := \sum_{n \ge 1} E_{n} z^{n-1}$ and $\tilde{E}(z) := \sum_{n \ge 0} \tilde{E}_{n} z^{n}$ one obtains that $E(z)$ satisfies the following second order non-linear differential equation:
\begin{equation}\label{eqn:DEQ_binary_path-number}
  E''(z) = \frac{1}{1-z} E'(z) + E(z) E'(z), \quad E(0)=1, \quad E'(0)=2.
\end{equation}
Differential equation \eqref{eqn:DEQ_binary_path-number} is not explicitly solvable; furthermore, the so-called Frobenius method to determine a singular expansion fails for $E(z)$. However, it is possible to apply the so-called psi-series method in the setting introduced in \cite{CheFerHwaMar2014}, i.e., assuming a logarithmic psi-series expansion of $E(z)$ when $z$ lies near the (unique) dominant singularity $\rho$ on the positive real axis (which, according to Pringsheim's theorem, exists and due to growth bounds satisfies $0 < \rho < 1$). We will here only give a sketch to identify the kind of singularity via the so-called ARS method for ordinary differential equations and to determine the asymptotic behaviour of $E(z)$ around $\rho$, whereas we refer in questions concerning the analytic continuation of solutions of \eqref{eqn:DEQ_binary_path-number} and the analyticity (i.e., absolute convergence) of the stated psi-series to the seminal work \cite{CheFerHwaMar2014}, where a general method has been proposed and illustrated by many examples for differential equations with a logarithmic branch point as dominant movable singularity $\rho$. We also do not pursue the task of determining a more precise numeric value for $\rho$.

\begin{itemize}
\item \emph{Leading order analysis:} we first assume that the solution of \eqref{eqn:DEQ_binary_path-number} admits a formal Laurent expansion around (a cut-disk of) the dominant singularity $\rho$ with the behaviour $E(z) \sim c_{0} (1-z/\rho)^{-\alpha}$. Setting $Z := 1-z/\rho$ and balancing the dominant terms in the differential equation yields
\begin{equation*}
  (-\alpha-1) Z^{-\alpha-2} + \rho c_{0} Z^{-2\alpha-1} = 0,
\end{equation*}
which implies $\alpha=1$ and furthermore $c_{0}= \frac{2}{\rho}$.

\item \emph{Resonance analysis:} we examine whether $\rho$ is a pole and thus the Frobenius method would be applicable. Let us assume $E(z)$ admits a local expansion $E(z) = \sum_{j \ge 0} c_{j} (1-z/\rho)^{j-1}$, with $c_{0}=\frac{2}{\rho}$. Plugging this form into \eqref{eqn:DEQ_binary_path-number} and equating coefficients yields the following recurrence for the coefficients $c_{j}$:
\begin{align*}
  & \quad (1-\rho) (j+1) (j-2) c_{j} =\\
	& \qquad -\rho(j^{2}-2j-2)c_{j-1} - \rho (1-\rho) \sum_{1 \le \ell \le j-1} (\ell-1) c_{\ell} c_{j-\ell} - \rho^{2} \sum_{1 \le \ell \le j-2} (\ell-1) c_{\ell} c_{j-1-\ell}.
\end{align*}
The left hand side is annihilated for $j=2$, which is called in this context a positive resonance and this value has to be examined further to decide whether the Frobenius method might work.

\item \emph{Compatibility:} the resonance $2$ is compatible if for $j=2$ also the right hand side of above equation vanishes. However, in our case this would require $c_{1}=0$, which does not hold as the correct value $c_{1}=\frac{1}{\rho-1}$ can be computed easily from above recurrence. 
\end{itemize}
Thus the solution of the differential equation does not admit a Laurent expansion around $\rho$; instead, following \cite{CheFerHwaMar2014}, a logarithmic psi-series expansion of the following form is proposed:
\begin{equation*}
  E(z) = \sum_{j \ge 0} Z^{j-1} \sum_{0 \le \ell \le \lfloor \frac{j}{2} \rfloor} c_{j,\ell} (\log Z)^{\ell}, \quad \text{with $Z=1-z/\rho$}.
\end{equation*}
Plugging the psi-series expansion into \eqref{eqn:DEQ_binary_path-number} and equating coefficients yields (with $c_{2,0}$ a certain constant):
\begin{equation*}
\begin{split}
  E(z) & = \frac{2}{\rho \, (1-z/\rho)} + \frac{1}{\rho-1} + c_{2,0} \big(1-z/\rho\big) - \frac{2 \rho \, \big(1-z/\rho\big) \log(1-z/\rho)}{3(\rho-1)^{2}} - \frac{\rho^{2} \, \big(1-z/\rho\big)^{2}}{2(\rho-1)^{2}}\\
	& \quad \mbox{} + \mathcal{O}\Big(\big(1-z/\rho\big)^{3} \log^{2}\big(1-z/\rho\big)\Big).
\end{split}
\end{equation*}
The stated result for $\mathbb{E}(P_{n}) = [z^{n-1}] E(z)$ follows by applying basic singularity analysis.
\end{proof}
As suggested by a referee, the first two terms in the asymptotic expansion of $E_{n}:= \mathbb{E}(P_{n})$ as stated in Theorem~\ref{thm:Binary_ExpNumberPaths} could also be derived directly from the recurrence
\begin{equation*}
  E_{n} = \frac{1}{(n-1)(n-2)} \sum_{k=2}^{n-1} (k-1) E_{k} (1+E_{n-k}), \quad n \ge 3, \quad E_{1}=1, E_{2}=2,
\end{equation*}
which is obtained from \eqref{eqn:DEQ_binary_path-number} by extracting coefficients.

\section{Generalizing the Bernoulli model: nonuniform duplication rules\label{sec:Bernoulli_non-uniform_rules}}

The stochastic growth rule in the Bernoulli model introduced in Section~\ref{ssec:Bernoulli} applied to a series-parallel network of size $n-1$ consists of two parts: first an edge $j=(x,y)$ is chosen amongst all edges of the network uniformly at random, and second this edge is duplicated according to a Bernoulli experiment, namely, either with probability $p$, $0 < p <1$, in a parallel way by inserting an additional edge $n=(x,y)$ right to $j$ into the graph, or otherwise in a serial way by replacing the (former) edge $j=(x,y)$ by edges $j=(x,z)$ and $n=(z,y)$, with $z$ a new node.
In order to generalize this Bernoulli model it seems natural to ask about alternative non-uniform selection rules for the choice of the edge in the first step and the influence on the structure of the generated graph.

We introduce here two such rules, where both have in common that the probability that a certain edge $j$ in the network is selected to ``attract'' the new edge $n$ depends on the number of edges that $j$ already has attracted in the past. For the first rule we assume that the probability that the new edge $n$ is attracted by edge $j$ is proportional to one plus the number of edges that have been already attracted by $j$, thus this rule might be called ``preferential attraction model''. For the second rule we assume that each edge can attract at most two edges during the whole growth process, i.e., after attracting the second edge it becomes saturated. To be more precise, we assume that the probability that the new edge $n$ is attracted by edge $j$ is proportional to $2-l$, with $l$ the number of edges that have been already attracted by $j$, thus this rule might be called ``saturation model''. We note that for both models the second part of the stochastic growth rule, i.e., the application of the ``Bernoulli edge-duplication rule'' to the attracted edge, is carried out exactly as for the uniform Bernoulli model.

The growth of series-parallel networks under these rules correspond to the growth of two important random increasing tree models, namely plane (or plane-oriented) recursive trees and binary increasing trees, respectively, see., e.g., \cite{PanPro2007}. Namely, the probability that the new node $n$ is attached to node $j$ (of a randomly chosen tree of order $n-1$) is proportional to one plus the out-degree of $j$ for plane recursive trees and proportional to $2-l$, with $l$ the out-degree of node $j$, for binary increasing trees. Again, in order to keep the information concerning the kind of duplication of the selected edge in the considered series-parallel network, in the corresponding tree model we colour the edge incident to $n$ blue for a parallel doubling and red for a serial doubling. The resulting edge-coloured increasing tree structures can be described also in a combinatorial way via formal specifications. Namely, the combinatorial family of edge-coloured plane recursive trees is given by
\begin{equation}\label{eqn:planerec_FEQ}
  \mathcal{T} = \mathcal{Z}^{\Box} \ast \textsc{SEQ}(\{B\} \times \mathcal{T} + \{R\} \times \mathcal{T}),
\end{equation}
whereas the family of edge-coloured binary increasing trees satisfies
\begin{equation}\label{eqn:binaryinc_FEQ}
  \mathcal{T} = \mathcal{Z}^{\Box} \ast (\{\epsilon\} + \{B\} \times \mathcal{T} + \{R\} \times \mathcal{T})^{2},
\end{equation}
with $B$ and $R$ markers. In order to get the right probability model we will assume that each marker $B$ gets the (multiplicative) weight $p$ and each marker $R$ the weight $q=1-p$. Note that the number $T_{n}$ of trees of order $n \ge 1$ in these families, when we forget about the colour of the edges, are given by $T_{n} = (2n-3)!! = \frac{(2n-2)!}{2^{n-1} (n-1)!}$ and $T_{n} = n!$ for plane recursive trees and binary increasing trees, respectively (see, e.g., \cite{FlaSed2009}).

\medskip

Adapting the combinatorial approach used in Section~\ref{sec:Bernoulli_model} to analyze the (uniform) Bernoulli model a treatment of quantities under these nonuniform duplication rules could be given. In order to describe the influence of the different growth rules to the structure of the series-parallel network and to compare it with the original Bernoulli model, in the following we state limiting distribution results for the degree $D_{n}$ of the source in a series-parallel network of size $n$ under the preferential attraction model and the saturation model, respectively. A proof of these results can be found in the appendix.% of an extended version on the arXiv.

For the preferential attraction model we get the following characterization of the limiting distribution of $D_{n}$.
Here $S_{\alpha}$ denotes a positive stable random variable with Laplace transform
\begin{equation*}
  \mathbb{E}(e^{-t S_{\alpha}}) = e^{-t^{\alpha}}, \quad \text{with $0 < \alpha < 1$}.
\end{equation*}
\begin{theorem}\label{the:planerec_thm}
  The degree $D_{n}$ of the source in a randomly chosen series-parallel network of size $n$ generated by the preferential attraction model converges after scaling, for $n \to \infty$, in distribution to the negative power of a positive stable random variable:
\begin{equation*}
  \frac{(p+1)^{2}}{p 2^{p+1} n^{\frac{p+1}{2}}} D_{n} \xrightarrow{(d)} D = D(p), \quad \text{with} \quad D \stackrel{(d)}{=} S_{\frac{2p}{p+1}}^{-p},
\end{equation*}
where $D$ is also characterized by its sequence of $r$-th integer moments:
\begin{equation*}
  \mathbb{E}(D^{r}) = \frac{\Gamma\big(r(\frac{p+1}{2}) +1\big)}{\Gamma(rp+1)}, \quad r \ge 0.
\end{equation*}
\end{theorem}

\medskip

For the saturation model we obtain the following limiting behaviour for $D_{n}$.
\begin{theorem}\label{the:binaryinc_thm}
  The degree $D_{n}$ of the source in a randomly chosen series-parallel network of size $n$ generated by the saturation model has, dependent on the probability $p$ of occurrences of parallel edge-duplications in the growth rule, the following limiting distribution behaviour.
\begin{itemize}
\item For $0 < p \le \frac{1}{2}$, the r.v.\ $D_{n}$ converges, for $n \to \infty$, in distribution to a discrete limit $D$, $D_{n} \xrightarrow{(d)} D=D(p)$, which is characterized via the following probability mass function:
\begin{equation*}
  \mathbb{P}\{D = m\} = \frac{1}{m+1} \binom{2m}{m} p^{m-1} (1-p)^{m+1}, \quad \text{for $m \ge 1$}.
\end{equation*}
\item For $\frac{1}{2} < p < 1$, the r.v.\ $D_{n}$ converges after suitable scaling, for $n \to \infty$, in distribution to the product of a Bernoulli distribution and a Mittag-Leffler distribution (i.e., the mixture of a Mittag-Leffler distribution and the distribution of a degenerate r.v.\ $0$):
\begin{equation*}
  \frac{(2p-1)^2}{p^{2}} \cdot \frac{D_{n}}{n^{2p-1}} \xrightarrow{(d)} D=D(p),
\end{equation*}
with
\begin{equation*}
  D \stackrel{(d)}{=} \emph{\text{Bernoulli}}\big(\frac{2p-1}{p^{2}}\big) \cdot \emph{\text{Mittag-Leffler}}(2p-1).
\end{equation*}
\end{itemize}
\end{theorem}

\section{Generalizing the binary model: the $b$-ary model\label{sec:bary_saturation_rules}}

A natural growth rule for series-parallel networks generalizing the binary model is obtained when assuming that each node in the network may have an out-degree at most $b$, with $b \ge 2$ a fixed integer. Namely, after selecting an edge $j=(x,y)$ in a series-parallel network of size $n-1$ uniformly at random, the decision which kind of edge-duplication is applied to $j$ is determined by the out-degree $d^{+}(x)$ of $x$: if $d^{+}(x)<b$ then a parallel doubling by inserting an additional edge $n=(x,y)$ right to edge $j$ is carried out, whereas otherwise, if $d^{+}(x)=b$ and thus $x$ is saturated, a serial doubling is done, where (former) edge $j=(x,y)$ is replaced by edges $j=(x,z)$ and $n=(z,y)$, with $z$ a new node. This uniform $b$-ary saturation edge-duplication rule will be denoted by ``$b$-ary model'' for short.

As for the binary model the growth of series-parallel networks under the $b$-ary model can be captured via bucket recursive trees, but with a maximal bucket size $b$, where nodes can hold up to $b$ labels. Here, in step $n$ each of the $n-1$ labels of a bucket recursive tree of order $n-1$ attracts the new label $n$ with equal probability, let us assume label $j$ contained in node $x$ is chosen. If node $x$ is saturated and thus already contains $b$ labels, then a new node containing label $n$ will be attached to $x$ as new child associated with label $j$, whereas otherwise, label $n$ will be inserted into node $x$ (let us assume right to label $j$).

A combinatorial top-down description of bucket recursive trees with bucket size $b$ (and even more general bucket increasing tree models) is given in \cite{KubPan2010}, where it has been shown also that both descriptions are equivalent. When we denote this combinatorial family by $\mathcal{B} := \mathcal{B}_{b}$, a formal description might be stated as follows:
\begin{equation}\label{eqn:formaleqn_bbucket}
\begin{split}
  \mathcal{B} & = \raisebox{-0.5ex}{\includegraphics[height=2.3ex]{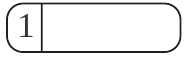}} \; + \;
  \raisebox{-0.5ex}{\includegraphics[height=2.3ex]{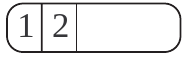}} \; + \; 
	2 \cdot \raisebox{-0.5ex}{\includegraphics[height=2.3ex]{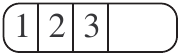}} \; + \; \cdots \; + \;
  (b-2)! \cdot \raisebox{-0.5ex}{\includegraphics[height=2.3ex]{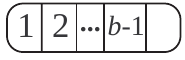}}\\
	& \quad \mbox{} + \; (b-1)! \cdot \raisebox{-0.5ex}{\includegraphics[height=2.3ex]{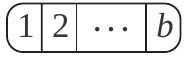}} \times \big(\textsc{Set}\big(\mathcal{B}\big)\big)^{b},
\end{split}
\end{equation}
where $\raisebox{-0.5ex}{\includegraphics[height=2.3ex]{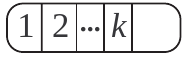}}$ denotes a bucket containing $k$ labels and $\times$ the Cartesian product. Note that the factor $(k-1)!$ for a bucket containing $k$ labels is stemming from the fact that there are $(k-1)!$ possibilities of generating such a bucket (e.g., whether label $3$ has been attracted by label $1$ or $2$ yields the buckets $\raisebox{-0.5ex}{\includegraphics[height=2.3ex]{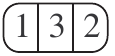}}$ and $\raisebox{-0.5ex}{\includegraphics[height=2.3ex]{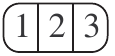}}$, respectively), but for our purpose it suffices to identify each of these instances.

Of course, from above stochastic description it follows immediately that there are $T_{n} := T_{n}^{[b]} = (n-1)!$ different bucket recursive trees of order $n$. Using the combinatorial description this result can be deduced as follows: introducing the generating function $T(z) := T^{[b]}(z) = \sum_{n \ge 1} T_{n} \frac{z^{n}}{n!}$, above formal recursive equation \eqref{eqn:formaleqn_bbucket} yields, by an application of the symbolic method and taking into account the initial values $T_{0}=0$ and $T_{k}=(k-1)!$, for $1 \le k \le b-1$, the differential equation
\begin{equation*}
  T^{(b)}(z) = (b-1)! \cdot e^{b T(z)}, \qquad T(0)=0, \quad T^{(k)}(0) = (k-1)!, \; \text{for $1 \le k \le b-1$}.
\end{equation*}
It can be checked easily that the solution of this equation is given by $T(z) = \log\big(\frac{1}{1-z}\big)$, thus also showing $T_{n} = (n-1)!$, for $n \ge 1$.

For the combinatorial analysis of series-parallel networks generated by the $b$-ary model it is important that the link given in Section~\ref{ssec:Binary} between the decomposition of a bucket recursive trees into the root node and the $b=2$ forests of subtrees attached to it and the subblock structure of the corresponding series-parallel network is taken over from $b=2$ to general $b$ as is illustrated in Figure~\ref{fig:link_decomposition_buckettree_network_bary}.
\begin{figure}
\begin{center}
\parbox{6.2cm}{\includegraphics[width=6cm]{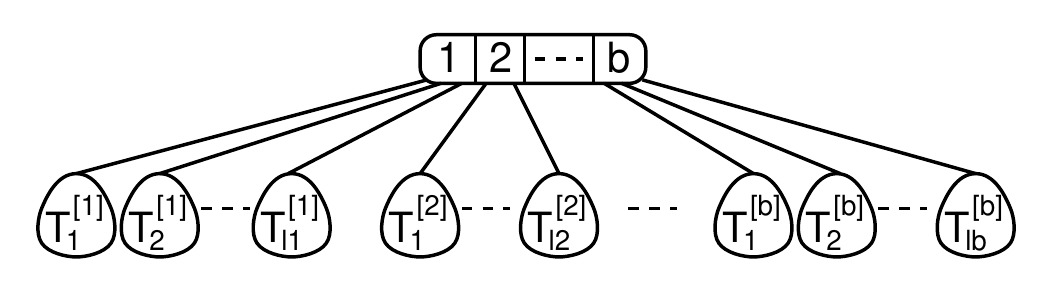}} \quad $\Longleftrightarrow$ \quad
\parbox{2.7cm}{\includegraphics[width=2.5cm]{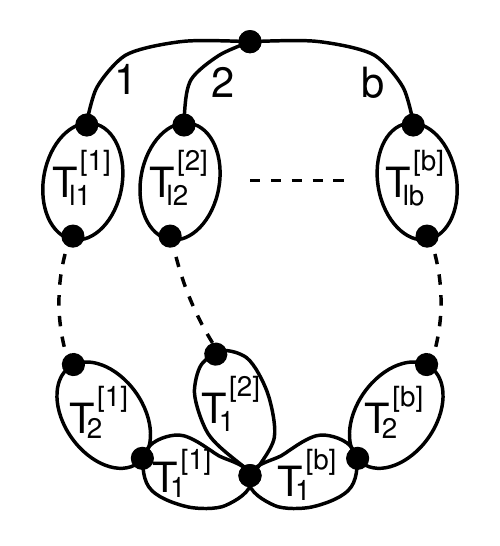}}
\end{center}
\vspace*{-3mm}
\caption{Decomposition of a bucket recursive tree $T$ into its root and the $b$ forests attached to the labels of the root and the subblock structure of the corresponding $b$-ary network.\label{fig:link_decomposition_buckettree_network_bary}}
\end{figure}

\medskip

The combinatorial approach used in Section~\ref{sec:Binary_model} for the analysis of the binary model can, at least in principle, be extended for a treatment of quantities in the $b$-ary model; however, computations are considerably more involved. In the following we only state a result for the length $L_{n}$ of a random source-to-sink path in a random series-parallel network of size $n$ under the $b$-ary model, where we restrict ourselves to a study of the expectation $\mathbb{E}(L_{n})$. A proof of the following theorem can be found in the appendix.% of an extended version on the arXiv.

We use here the abbreviation $H_{x+m} - H_{x} := \sum_{k=1}^{m} \frac{1}{x+k}$, for $m \in \mathbb{N}$ and $x \in \mathbb{C} \setminus \{-1, -2, \dots\}$ and, for a better readability, suppress the (obvious) dependence on $b$ in the quantity studied, i.e., $L_{n} := L_{n}^{[b]}$.
\begin{theorem}\label{thm:ExpLengthRandomPath_b-ary}
  The expectation $\mathbb{E}(L_{n})$ of the length $L_{n}$ of a random source-to-sink path in a random series-parallel network of size $n$ generated by the $b$-ary model is given as follows:
	\begin{equation*}
	  \mathbb{E}(L_{n}) = \sum_{i=1}^{b} \frac{1}{1+\lambda_{i} (H_{\lambda_{i}+b-1} - H_{\lambda_{i}})} \cdot \binom{n+\lambda_{i}-1}{n-1}
		\sim \frac{1}{1+\lambda_{1}(H_{\lambda_{1}+b-1}-H_{\lambda_{1}})} \cdot \frac{n^{\lambda_{1}}}{\Gamma(\lambda_{1}+1)},
	\end{equation*}
	where $\lambda_{i}$, $1 \le i \le b$, denote the $b$ different roots of the characteristic equation $\lambda^{\overline{b}} = (b-1)!$ with $\lambda_{1} \in (0,1)$ the unique positive real root of this equation.
\end{theorem}

\begin{appendix}

\section{Proof of Theorem~\ref{the:planerec_thm} concerning the preferential attraction model}

For a better readability we divide the proof into smaller parts, from which we combine the main theorem.
\begin{proposition}\label{prop:planerec_Fzv}
  The bivariate generating function
	\begin{equation*}
	  F(z,v) := \sum_{n \ge 1} \sum_{m \ge 1} T_{n} \mathbb{P}\{D_{n}=m\} \frac{z^{n}}{n!} v^{m}, \quad \text{with} \quad T_{n} = \frac{(2n-2)!}{2^{n-1} (n-1)!},
	\end{equation*}
	of the probabilities $\mathbb{P}\{D_{n}=m\}$ satisfies the first order non-linear differential equation
	\begin{equation}\label{eqn:DEQ_Fzv_planerec}
	  F'(z,v) = \frac{v}{p+(1-p)\sqrt{1-2z} - p F(z,v)}, \quad F(0,v)=0.
	\end{equation}
\end{proposition}
\begin{proof}
  We adapt the generating functions proof of the distribution of $D_{n}$ given for the Bernoulli model in Section~\ref{ssec:Bernoulli_Degree}. In the tree model, $D_{n}$ measures the order of the blue subtree, i.e., the number of nodes that can be reached from the root node by taking only blue edges, in a random edge-coloured plane recursive tree of order $n$. Furthermore we introduce the r.v.\ $D_{n,k}$, whose distribution is given as the conditional distribution $D_{n} | \{\text{the tree has exactly $k$ blue edges}\}$ as well as the trivariate generating function
\begin{equation*}
  F(z,u,v) := \sum_{n \ge 1} \sum_{0 \le k \le n-1} \sum_{m \ge 1} T_{n} \binom{n-1}{k} \mathbb{P}\{D_{n,k}=m\} \frac{z^{n}}{n!} u^{k}v^{m},
\end{equation*}
with $T_{n}$ the number of plane recursive trees of order $n$. We also require the auxiliary function $N(z,u) := \sum_{n \ge 1} \sum_{0 \le k \le n-1} T_{n} \binom{n-1}{k} \frac{z^{n}}{n!} u^{k} = \frac{1}{1+u} \left(1-\sqrt{1-2z(1+u)}\right)$, i.e., the exponential generating function of the number of edge-coloured plane recursive trees of order $n$ with exactly $k$ blue edges. Using the decomposition of a tree into the root node and its branches according to \eqref{eqn:planerec_FEQ} with considerations completely analogous to the ones given in Section~\ref{ssec:Bernoulli_Degree} show then the following first order non-linear differential equation for $F:=F(z,u,v)$:
\begin{equation*}
  F' = \frac{v}{1-(N+uF)} = \frac{(1+u) v}{u + \sqrt{1-2z(1+u)} - (1+u)uF},
\end{equation*}
with initial condition $F(0,u,v)=0$. The bivariate generating function $F(z,v)$ can be obtained from $F(z,u,v)$ via the relation
$F(z,v) = \frac{1}{q} F(qz, \frac{p}{q}, v)$, which, after simple manipulations, shows the proposition.
\end{proof}

In order to determine the asymptotic behaviour of the integer moments of $D_{n}$ we require the following lemma, where we use the abbreviations $\partial_{v}$ for the differential operator w.r.t.\ $v$ and $V$ for the operator evaluating at $v=1$.
\begin{lemma}\label{lem:planerec_Mrz}
  Let $M_{r}(z) := V \partial_{v}^{r} F(z,v) = \sum_{n \ge 1} T_{n} \mathbb{E}(D_{n}^{\underline{r}}) \frac{z^{n}}{n!}$ be the generating function of the $r$-th factorial moments of $D_{n}$. Then the local behaviour of $M_{r}(z)$ in a complex neighbourhood of the unique dominant singularity $z=\frac{1}{2}$ is given as follows:
	\begin{equation*}
	  M_{r}(z) \sim \frac{\alpha_{r}}{(1-2z)^{\frac{rp}{2} + \frac{r-1}{2}}}, \quad r \ge 1,
	\end{equation*}
	where the sequence $\alpha_{r}$ of coefficients satisfies the recurrence (with $\alpha_{0}=-1$ and $\alpha_{1}=\frac{1}{p+1}$):
	\begin{equation}\label{eqn:planerec_alphar_rec}
	  \alpha_{r} = \frac{p}{(r-1)(p+1)} \sum_{0 \le r_{0} \le r-2} \sum_{0 \le r_{1} \le r-1-r_{0}} \binom{r-1}{r_{0}, r_{1}, r-1-r_{0}-r_{1}}
		\alpha_{r_{0}+1} \beta_{r_{1}} \beta_{r-1-r_{0}-r_{1}}, \quad r \ge 2,
	\end{equation}
	and where the auxiliary sequence $\beta_{r}$ is defined via $\beta_{r} := (rp+r-1) \alpha_{r}$.
\end{lemma}
\begin{proof}
  This lemma can be shown inductively, where we introduce as auxiliary functions\newline
	$\tilde{M}_{r}(z) := V \partial_{v}^{r} \left(\frac{1}{p+(1-p)\sqrt{1-2z} -p F(z,v)}\right)$ and prove in parallel that the local behaviour of $\tilde{M}_{r}(z)$ around the unique dominant singularity $z=\frac{1}{2}$ is given as follows:
	\begin{equation*}
	  \tilde{M}_{r}(z) \sim \frac{\beta_{r}}{(1-2z)^{\frac{rp}{2}+\frac{r+1}{2}}}, \quad r \ge 0,
	\end{equation*}
	with $\beta_{r} = (rp+r-1) \alpha_{r}$. For $r=0$ we obtain $M_{0}(z) = F(z,1) = \sum_{n \ge 1} T_{n} \frac{z^{n}}{n!} = 1-\sqrt{1-2z}$ and $\tilde{M}_{0}(z) = \frac{1}{p+(1-p)\sqrt{1-2z} - p M_{0}(z)} = \frac{1}{\sqrt{1-2z}}$. The unique dominant singularity of both functions is at $z=\frac{1}{2}$ and the local behaviour of $\tilde{M}_{0}(z)$ around this singularity is as stated with $\beta_{0}=1$, since we define $\alpha_{0} = -1$. Next we consider $r \ge 1$ and apply the operator $V \partial_{v}^{r}$ to the differential equation \eqref{eqn:DEQ_Fzv_planerec}, which gives the connection
	\begin{equation}\label{eqn:DEQ_Mrz_tMrz_planerec}
	  M_{r}'(z) = \tilde{M}_{r}(z) + r \tilde{M}_{r-1}(z).
	\end{equation}
	Moreover, when applying $V \partial_{v}^{r}$ to $\frac{1}{p+(1-p)\sqrt{1-2z}-pF(z,v)}$ we get
	\begin{align*}
	  \tilde{M}_{r}(z) & = V \partial_{v}^{r-1} \frac{p \partial_{v} F(z,v)}{(p+(1-p)\sqrt{1-2z}-pF(z,v))^{2}}\\
		& = p \sum_{r_{0} +r_{1}+r_{2}=r-1} \binom{r-1}{r_{0},r_{1},r_{2}} M_{r_{0}+1}(z) \tilde{M}_{r_{1}}(z) \tilde{M}_{r_{2}}(z)\\
		& = \frac{p}{1-2z} M_{r}(z) + p \sum_{0 \le r_{0} \le r-2} \sum_{0 \le r_{1} \le r-1-r_{0}} {\textstyle \binom{r-1}{r_{0}, r_{1}, r-1-r_{0}-r_{1}}} M_{r_{0}+1}(z) \tilde{M}_{r_{1}}(z) \tilde{M}_{r-1-r_{0}-r_{1}}(z).
	\end{align*}
	Let us consider the instance $r=1$ separately, where we get after simple manipulations the equations
	\begin{equation*}
	  M_{1}'(z) = \frac{p}{1-2z} M_{1}(z) + \frac{1}{\sqrt{1-2z}}, \quad M_{1}(0)=0, \quad \tilde{M}_{1}(z) = \frac{p}{1-2z} M_{1}(z),
	\end{equation*}
	which easily yield the following explicit solutions:
	\begin{equation*}
	  M_{1}(z) = \frac{1}{(1+p)(1-2z)^{\frac{p}{2}}} - \frac{\sqrt{1-2z}}{1+p}, \quad \tilde{M}_{1}(z) = \frac{p}{(1+p)(1-2z)^{\frac{p}{2}+1}} - \frac{p}{(1+p)\sqrt{1-2z}}.
	\end{equation*}
	Thus, also these functions have their unique dominant singularities at $z=\frac{1}{2}$ and the local behaviour around this singularity is as stated, i.e., $\alpha_{1} = \frac{1}{1+p}$ and $\beta_{1} = \frac{p}{1+p} = p \alpha_{1}$.
	
	Now we turn to general $r \ge 2$; from above computations we deduce that $M_{r}(z)$ is defined via the first order linear differential equation
	\begin{equation}\label{eqn:planerec_degree_Mrz_DEQ}
	  M_{r}'(z) = \frac{p}{1-2z} M_{r}(z) + S_{r}(z), \quad M_{r}(0)=0,
	\end{equation}
	with inhomogeneous part
	\begin{equation*}
	  S_{r}(z) := r \tilde{M}_{r-1}(z) + p \sum_{0 \le r_{0} \le r-2} \sum_{0 \le r_{1} \le r-1-r_{0}} {\textstyle \binom{r-1}{r_{0}, r_{1}, r-1-r_{0}-r_{1}}} M_{r_{0}+1}(z) \tilde{M}_{r_{1}}(z) \tilde{M}_{r-1-r_{0}-r_{1}}(z).
	\end{equation*}
	The solution of this differential equation can be obtained by standard methods and can be written as follows:
	\begin{equation}\label{eqn:planerec_degree_Mrz_sol}
	  M_{r}(z) = \frac{1}{(1-2z)^{\frac{p}{2}}} \int_{0}^{z} (1-2t)^{\frac{p}{2}} S_{r}(t) dt.
	\end{equation}
	From representation \eqref{eqn:planerec_degree_Mrz_sol} it is immediate that, assuming $M_{j}(z)$ and $\tilde{M}_{j}(z)$ have their unique dominant singularities at $z=\frac{1}{2}$, for $j < r$, this also holds for $S_{r}(z)$, $M_{r}(z)$ and, by taking into account \eqref{eqn:DEQ_Mrz_tMrz_planerec}, for $\tilde{M}_{r}(z)$. Moreover, when we assume the stated local behaviour of $M_{j}(z)$ and $\tilde{M}_{j}(z)$, for all $j < r$, in a neighbourhood of the dominant singularity, we obtain after straightforward computations the following local behaviour of $S_{r}(z)$:
	\begin{equation*}
	  S_{r}(z) \sim p \sum_{0 \le r_{0} \le r-2} \sum_{0 \le r_{1} \le r-1-r_{0}} \binom{r-1}{r_{0}, r_{1}, r-1-r_{0}-r_{1}} \frac{\alpha_{r_{0}+1} \beta_{r_{1}} \beta_{r-1-r_{0}-r_{1}}}{(1-2z)^{\frac{rp}{2}+\frac{r+1}{2}}}.
	\end{equation*}
	Applying closure properties concerning singular integration we deduce from it:
	\begin{align*}
	  & M_{r}(z) = \frac{1}{(1-2z)^{\frac{p}{2}}} \int_{0}^{z} (1-2t)^{\frac{p}{2}} S_{r}(t) dt\\
		& \quad \sim \frac{1}{(1-2z)^{\frac{rp}{2}+\frac{r-1}{2}}} \cdot \frac{p}{(r-1)(p+1)} \sum_{0 \le r_{0} \le r-2} \sum_{0 \le r_{1} \le r-1-r_{0}} {\textstyle \binom{r-1}{r_{0}, r_{1}, r-1-r_{0}-r_{1}}} \alpha_{r_{0}+1} \beta_{r_{1}} \beta_{r-1-r_{0}-r_{1}},
	\end{align*}
	thus the local behaviour around $z=\frac{1}{2}$ given above also holds for $M_{r}(z)$ with $\alpha_{r}$ obtained recursively.
	Furthermore, using \eqref{eqn:DEQ_Mrz_tMrz_planerec} and singular differentiation, we obtain
	\begin{equation*}
	  \tilde{M}_{r}(z) = M_{r}'(z) - r \tilde{M}_{r-1}(z) \sim M_{r}'(z) \sim \frac{(rp+r-1) \alpha_{r}}{(1-2z)^{\frac{rp}{2}+\frac{r+1}{2}}},
	\end{equation*}
	i.e., the stated local behaviour of $\tilde{M}_{r}(z)$ with $\beta_{r} = (rp+r-1) \alpha_{r}$ is valid also for $r \ge 2$. 
\end{proof}

Interestingly, the sequence of coefficients $\alpha_{r}$ defined recursively via \eqref{eqn:planerec_alphar_rec}, which occurs in the local behaviour of the generating functions $M_{r}(z)$ defined in Lemma~\ref{lem:planerec_Mrz}, admits a nice explicit formula. We mention that first this formula has been guessed from factorizations of $\alpha_{r}$, for $r$ small. In the following lemma we state this result together with a generating functions proof of it.
\begin{lemma}\label{lem:planerec_alphar}
  The numbers $\alpha_{r}$ defined recursively according to \eqref{eqn:planerec_alphar_rec} are given by the following explicit formula:
	\begin{equation*}
	  \alpha_{r} = \frac{(r-1)! p^{r-1} \binom{r(p+1)-2}{r-1}}{(p+1)^{2r-1}}, \quad r \ge 1.
	\end{equation*}
\end{lemma}
\begin{proof}
We treat recurrence \eqref{eqn:planerec_alphar_rec} via generating functions and to this aim we introduce $A(z) := \sum_{r \ge 0} \alpha_{r} \frac{z^{r}}{r!}$ and $B(z) := \sum_{r \ge 0} \beta_{r} \frac{z^{r}}{r!}$. Straightforward computations yield the relations
\begin{equation}\label{eqn:planerec_AzBz_DEQ}
  (p+1) z A''(z) = p A'(z) (B^{2}(z)-1), \quad B(z) = (p+1) z A'(z) - A(z).
\end{equation}
It turns out to be advantageous to consider $\tilde{A}(z) = 1 + A(z)$, which, according to \eqref{eqn:planerec_AzBz_DEQ} and after simple manipulations, is characterized via the second order (non-linear) differential equation
\begin{equation}\label{eqn:planerec_tildeAz_DEQ}
  (p+1) z \tilde{A}''(z) = p \tilde{A}'(z) \big((p+1)z \tilde{A}'(z)-\tilde{A}(z)\big) \cdot \big((p+1) z \tilde{A}'(z) - \tilde{A}(z) +2\big),
\end{equation}
with initial conditions $\tilde{A}(0)=0$ and $\tilde{A}'(0)=\frac{1}{p+1}$.
We claim that the solution $\tilde{A} = \tilde{A}(z)$ of \eqref{eqn:planerec_tildeAz_DEQ} is given by the solution of the functional equation
\begin{equation}\label{eqn:planerec_tildeAz_FEQ}
  \tilde{A} = \frac{z}{(p+1) \left(1-\frac{p \tilde{A}}{p+1}\right)^{p}}.
\end{equation}
From \eqref{eqn:planerec_tildeAz_DEQ} we get after some computations the following formul{\ae} for the first two derivatives:
\begin{equation*}
  \tilde{A}'(z) = \frac{\tilde{A} \left(1-\frac{p \tilde{A}}{p+1}\right)}{z(1-p\tilde{A})}, \quad \tilde{A}''(z) = \frac{p^{2} \tilde{A}^{2} (2-p \tilde{A}) \left(1-\frac{p \tilde{A}}{p+1}\right)}{z^{2} (p+1) (1-p\tilde{A})^{3}}.
\end{equation*}
Plugging these expressions into \eqref{eqn:planerec_tildeAz_DEQ} shows after simple manipulations that $\tilde{A}(z)$ defined via \eqref{eqn:planerec_tildeAz_FEQ} indeed solves above differential equation and also satisfies the given initial conditions.

Thus, according to $\alpha_{r} = r! [z^{r}] \tilde{A}(z)$, for $r \ge 1$, we only have to extract coefficients from \eqref{eqn:planerec_tildeAz_FEQ}, which can be done by a standard application of the Lagrange inversion formula (see, e.g., \cite{FlaSed2009}):
\begin{align*}
  \alpha_{r} & = r! [z^{r}] \tilde{A}(z) = r! \frac{1}{r} [\tilde{A}^{r-1}] \frac{1}{(p+1)^{r} \left(1-\frac{p \tilde{A}}{p+1}\right)^{p r}} = \frac{(r-1)! p^{r-1}}{(p+1)^{2r-1}} [\tilde{A}^{r-1}] \frac{1}{(1-\tilde{A})^{p r}}\\
	& = \frac{(r-1)! p^{r-1} \binom{r(p+1)-2}{r-1}}{(p+1)^{2r-1}}, \quad \text{for $r \ge 1$}.
\end{align*}
This completes the proof of the lemma.
\end{proof}

\begin{proof}[Proof of Theorem~\ref{the:planerec_thm}]
 From Lemma~\ref{lem:planerec_Mrz}  we obtain by an application of basic singularity analysis the following asymptotic behaviour of the coefficients of $M_{r}(z)$, for $r$ fixed and $n \to \infty$:
\begin{equation*}
  [z^{n}] M_{r}(z) \sim \alpha_{r} 2^{n} \frac{n^{\frac{rp}{2} + \frac{r-3}{2}}}{\Gamma(\frac{rp}{2} + \frac{r-1}{2})}.
\end{equation*}
Together with the asymptotic behaviour of the number of plane recursive trees
\begin{equation*}
  \frac{T_{n}}{n!} = \frac{1}{2^{n-1} n} \binom{2n-2}{n-2} \sim \frac{2^{n-1}}{\sqrt{\pi} n^{\frac{3}{2}}}
\end{equation*}
we obtain for the $r$-th integer moments of the r.v.\ $D_{n}$:
\begin{equation*}
  \mathbb{E}(D_{n}^{r}) \sim \mathbb{E}(D_{n}^{\underline{r}}) = \frac{n!}{T_{n}} [z^{n}] M_{r}(z) \sim \frac{\alpha_{r} 2 \sqrt{\pi} n^{\frac{r(p+1)}{2}}}{\Gamma(\frac{r(p+1)}{2}-\frac{1}{2})}.
\end{equation*}
Using the explicit formula for the numbers $\alpha_{r}$ given in Lemma~\ref{lem:planerec_alphar} as well as the duplication formula for the Gamma-function we proceed with
\begin{equation}\label{eqn:planerec_degree_mom_asy}
  \mathbb{E}(D_{n}^{r}) \sim \frac{p^{r} 2^{r (p+1)}}{(p+1)^{2r}} \cdot \frac{\Gamma(\frac{r(p+1)}{2} +1)}{\Gamma(rp+1)} \cdot n^{\frac{r(p+1)}{2}}, \quad \text{for $r \ge 1$}.
\end{equation}
From \eqref{eqn:planerec_degree_mom_asy} an application of the theorem of Fr\'echet and Shohat shows that after suitable scaling, $D_{n}$ converges for $n \to \infty$ in distribution to a r.v.\ $Y=Y(p)$, $\frac{(p+1)^{2} D_{n}}{p 2^{p+1} n^{\frac{p+1}{2}}} \xrightarrow{(d)} Y$, where $Y$ is characterized via the sequence of $r$-th integer moments: $\mathbb{E}(Y^{r}) = \frac{\Gamma(\frac{r(p+1)}{2}+1)}{\Gamma(rp+1)}$, for $r \ge 0$. Note that according to simple growth bounds of the moments they indeed uniquely characterize the distribution of $Y$. 

However, following considerations by Janson given in \cite{Janson2010}, we can give an alternative description of the limiting distribution. Namely, let $S_{\alpha}$ be a positive stable random variable with Laplace transform $\mathbb{E}(e^{-t S_{\alpha}}) = e^{-t^{\alpha}}$, with $0 < \alpha < 1$; then it holds $\mathbb{E}(S_{\alpha}^{-s}) = \frac{\Gamma(\frac{s}{\alpha}+1)}{\Gamma(s+1)}$, for $s > 0$. Thus, when defining $X:=X_{\alpha,\beta} = S_{\alpha}^{-\beta}$, with $0 < \alpha, \beta < 1$, the positive real moments of $X$ are given by $\mathbb{E}(X^{s}) = \mathbb{E}(S_{\alpha}^{-\beta s}) = \frac{\Gamma(\frac{\beta s}{\alpha}+1)}{\Gamma(\beta s + 1)}$, for $s > 0$. Thus, by setting $\alpha := \frac{2p}{p+1} < 1$ and $\beta := p <1$, we obtain that the $r$-th integer moments of $X_{\alpha, \beta}$ indeed coincide with the moments of $Y$ defined above, thus $Y \stackrel{(d)}{=} S_{\frac{2p}{p+1}}^{-p}$.
\end{proof}

\section{Proof of Theorem~\ref{the:binaryinc_thm} concerning the saturation model}

We partition the proof into smaller parts, from which the main theorem can be deduced easily.
\begin{proposition}\label{prop:binaryinc_Qzv}
  Let
	\begin{equation*}
	  F(z,v) := \sum_{n \ge 1} \sum_{m \ge 1} T_{n} \mathbb{P}\{D_{n}=m\} \frac{z^{n}}{n!} v^{m}, \quad \text{with} \quad T_{n} = n!,
	\end{equation*}
	be the bivariate generating function of the probabilities $\mathbb{P}\{D_{n}=m\}$ and
	\begin{equation*}
	  \tilde{Q}(z,v) := pv \left(p F(z,v) + \frac{1-pz}{1-z}\right)
	\end{equation*}
	a linear variant. Then $\tilde{Q}(z,v)$ satisfies the following first order Riccati differential equation:
	\begin{equation}\label{eqn:DEQ_Qzv_binaryinc}
	  \tilde{Q}'(z,v) = \tilde{Q}(z,v)^{2} + \frac{p(1-p)v}{(1-z)^{2}}, \quad \tilde{Q}(0,v) = pv.
	\end{equation}
\end{proposition}
\begin{proof}
  This result follows completely analogous to Proposition~\ref{prop:planerec_Fzv}. $D_{n}$ measures in the tree model the order of the blue subtree in a random edge-coloured binary increasing tree of order $n$. Furthermore we introduce the r.v.\ $D_{n,k}$, whose distribution is given as the conditional distribution $D_{n} | \{\text{the tree has exactly $k$ blue edges}\}$ as well as the trivariate generating function
\begin{equation*}
  F(z,u,v) := \sum_{n \ge 1} \sum_{0 \le k \le n-1} \sum_{m \ge 1} T_{n} \binom{n-1}{k} \mathbb{P}\{D_{n,k}=m\} \frac{z^{n}}{n!} u^{k}v^{m},
\end{equation*}
with $T_{n} = n!$ the number of binary increasing trees of order $n$.\newline
Let $N(z,u) := \sum_{n \ge 1} \sum_{0 \le k \le n-1} T_{n} \binom{n-1}{k} \frac{z^{n}}{n!} u^{k} = \frac{z}{1-(1+u)z}$ be the exponential generating function of the number of edge-coloured binary increasing trees of order $n$ with exactly $k$ blue edges. The decomposition of a tree into the root node and its branches according to \eqref{eqn:binaryinc_FEQ} yields the following first order non-linear differential equation for $F:=F(z,u,v)$:
\begin{equation*}
  F' = v (1+N+uF)^{2} = v \left(1+\frac{z}{1-(1+u)z}+uF\right)^{2},
\end{equation*}
with initial condition $F(0,u,v)=0$. The bivariate generating function $F(z,v)$ can be deduced from $F(z,u,v)$ via the relation
$F(z,v) = \frac{1}{q} F(qz, \frac{p}{q}, v)$, which satisfies the first order non-linear differential equation
\begin{equation}\label{eqn:binaryinc_Fzv_DEQ}
  F'(z,v) = v \left(1+\frac{(1-p)z}{1-z} + p F(z,v)\right)^{2}, \quad F(0,v)=0.
\end{equation}
The stated Riccati differential equation \eqref{eqn:DEQ_Qzv_binaryinc} for $\tilde{Q}(z,v)$ follows from \eqref{eqn:binaryinc_Fzv_DEQ} after simple computations.
\end{proof}

First, we will deduce from Proposition~\ref{prop:binaryinc_Qzv} the asymptotic behaviour of the probabilities $\mathbb{P}\{D_{n}=m\}$, for $m$ fixed and $n \to \infty$.
\begin{lemma}\label{lem:binaryinc_probDn_limit}
  It holds for every $m \ge 1$ fixed:
	\begin{equation*}
	  p_{m} := \lim_{n \to \infty} \mathbb{P}\{D_{n}=m\} = \frac{1}{m+1} \binom{2m}{m} p^{m-1} (1-p)^{m+1}.
	\end{equation*}
\end{lemma}
\begin{proof}
  Although the differential equation \eqref{eqn:DEQ_Qzv_binaryinc} admits an explicit solution we find it more convenient to extract inductively the asymptotic behaviour of the considered probabilities. To this aim we introduce the functions
	\begin{equation*}
	\tilde{N}_{m}(z) := \left.\frac{\partial^{m}}{\partial v^{m}} \tilde{Q}(z,v)\right|_{v=0}, \quad \text{for $m \ge 0$}.
	\end{equation*}
	According to the definition we obtain
	\begin{equation}\label{eqn:binaryinc_Nmz_prob}
	  \tilde{N}_{m+1}(z) = p^{2} (m+1)! \sum_{n \ge 1} \mathbb{P}\{D_{n}=m\} z^{n}, \quad \text{for $m \ge 1$}.
	\end{equation}
	It is apparent that $\tilde{N}_{0}(z)=0$. Differentiating \eqref{eqn:DEQ_Qzv_binaryinc} w.r.t.\ $v$ and evaluating at $v=0$ yields	$\tilde{N}_{1}'(z) = \frac{p(1-p)}{(1-z)^{2}}$ with initial condition $\tilde{N}_{1}(0)=p$, thus
	\begin{equation}\label{eqn:binaryinc_N1z_exp}
	  \tilde{N}_{1}(z) = \frac{p-zp^{2}}{1-z}.
  \end{equation}
	Furthermore, for $m \ge 2$ we obtain by differentiating \eqref{eqn:DEQ_Qzv_binaryinc} $m$ times w.r.t.\ $v$, evaluating at $v=0$, followed by an integration:
	\begin{equation}\label{eqn:binaryinc_Nmz_exp}
	  \tilde{N}_{m}(z) = \sum_{\ell=1}^{m-1} \binom{m}{\ell} \int_{0}^{z} \tilde{N}_{\ell}(t) \tilde{N}_{m-\ell}(t) dt.
	\end{equation}
	From equations \eqref{eqn:binaryinc_N1z_exp} and \eqref{eqn:binaryinc_Nmz_exp} it follows immediately that the unique dominant singularity of the functions $\tilde{N}_{m}(z)$, $m \ge 1$, is at $z=1$. Moreover, one can easily show that the local behaviour of $\tilde{N}_{m}(z)$ in a complex neighbourhood of $z=1$ is given by 
	\begin{equation}\label{eqn_binaryinc_Nmz_local}
	  \tilde{N}_{m}(z) \sim \frac{\eta_{m}}{1-z}, \quad \text{for $m \ge 1$},
	\end{equation}
	with certain constants $\eta_{m}$.
	Namely, from \eqref{eqn:binaryinc_N1z_exp} we get $\tilde{N}_{1}(z) \sim \frac{p(1-p)}{1-z}$, thus $\eta_{1} = p(1-p)$, whereas \eqref{eqn:binaryinc_Nmz_exp} yields for $m \ge 2$:
	\begin{equation*}
	  \tilde{N}_{m}(z) \sim \sum_{\ell=1}^{m-1} \int_{0}^{z} \binom{m}{\ell} \frac{\eta_{\ell}}{1-t} \frac{\eta_{m-\ell}}{1-t} dt
		= \frac{1}{1-z} \sum_{\ell=1}^{m-1} \binom{m}{\ell} \eta_{\ell} \eta_{m-\ell},
	\end{equation*}
	thus 
	\begin{equation}\label{eqn:binaryinc_etam_rec}
	  \eta_{m} = \sum_{\ell=1}^{m-1} \binom{m}{\ell} \eta_{\ell} \eta_{m-\ell}, \quad \text{for $m \ge 2$}.
	\end{equation}
	To treat recurrence \eqref{eqn:binaryinc_etam_rec} we introduce the generating function $E(z) := \sum_{m \ge 1} \eta_{m} \frac{z^{m}}{m!}$, which yields
	\begin{equation*}
	  E(z)^{2} - E(z) + \eta_{1} z = 0.
	\end{equation*}
	Taking into account $E(0)=0$ and the initital value $\eta_{1} = p(1-p)$ we get the solution
	\begin{equation*}
	  E(z) = \frac{1-\sqrt{1-4p(1-p)z}}{2}.
	\end{equation*}
	Thus the coefficients $\eta_{m} = m! [z^{m}] E(z)$ are given by
	\begin{equation}\label{eqn:binaryinc_etam_formula}
	  \eta_{m} = (m-1)! \binom{2m-2}{m-1} (p(1-p))^{m}, \quad m \ge 1.
	\end{equation}
	
	Taking into account \eqref{eqn:binaryinc_Nmz_prob}, \eqref{eqn_binaryinc_Nmz_local} and \eqref{eqn:binaryinc_etam_formula} basic singularity analysis shows the stated results.
\end{proof}

For $p \le \frac{1}{2}$ this lemma will be sufficient to characterize the limiting behaviour of $D_{n}$, whereas for $p > \frac{1}{2}$ we will use the method of moments. As a preliminary result, which already shows the different behaviour of $D_{n}$ depending on $p$, we give explicit and asymptotic results for the expectation. In particular, we want to remark that there is also a different limiting behaviour for the range $p<\frac{1}{2}$ and $p=\frac{1}{2}$, since for $p=\frac{1}{2}$ the $r$-th integer moments of the limit do not exist, whereas for $p<\frac{1}{2}$ these moments are characterizing the limit.
\begin{lemma}\label{lem:binarinc_expDn}
The expectation $\mathbb{E}(D_{n})$ of $D_{n}$ is given by the following exact formula (with $0 < p < 1$):
\begin{equation*}
  \mathbb{E}(D_{n}) = 
	\begin{cases} 
	  \frac{1}{1-2p} \left(1-\binom{n+2p-1}{n}\right), & \quad \text{for $p \neq \frac{1}{2}$},\\
		H_{n}, & \quad \text{for $p=\frac{1}{2}$}.
	\end{cases}
\end{equation*}
Thus, it has the following asymptotic behaviour:
\begin{equation*}
  \mathbb{E}(D_{n}) \sim
	\begin{cases}
	  \frac{1}{1-2p}, & \quad \text{for $p < \frac{1}{2}$},\\
		\log n, & \quad \text{for $p=\frac{1}{2}$},\\
		\frac{n^{2p-1}}{(2p-1) \Gamma(2p)}, & \quad \text{for $p > \frac{1}{2}$}.
	\end{cases}
\end{equation*}
\end{lemma}
\begin{proof}
  We introduce the function
	\begin{equation}\label{eqn:binaryinc_M1z_def}
	  \tilde{M}_{1}(z) := \left.\frac{\partial}{\partial v}\tilde{Q}(z,v)\right|_{v=1} = \frac{p}{1-z} + p^{2} \sum_{n \ge 1} \mathbb{E}(D_{n}) z^{n},
	\end{equation}
	where the link to the expectation follows easily from the definition of $\tilde{Q}(z,v)$ given in Proposition~\ref{prop:binaryinc_Qzv}.
	
  Differentiating the differential equation \eqref{eqn:DEQ_Qzv_binaryinc} w.r.t.\ $v$ and evaluating at $v=1$ yields after simple manipulations the following first order linear differential equation for $\tilde{M}_{1}(z)$:
	\begin{equation*}
	 \tilde{M}_{1}'(z) = \frac{2p}{1-z} \tilde{M}_{1}(z) + \frac{p(1-p)}{(1-z)^{2}}, \quad \tilde{M}_{1}(0)=p,
	\end{equation*}
	which, by standard methods, gives the following explicit solution:
	\begin{equation}\label{eqn:binaryinc_M1z_sol}
	  \tilde{M}_{1}(z) = 
		\begin{cases} 
		  \frac{p(1-p)}{1-2p} \frac{1}{1-z} - \frac{p^{2}}{1-2p} \frac{1}{(1-z)^{2p}}, & \quad \text{for $p \neq \frac{1}{2}$},\\
			\frac{1}{4(1-z)} \log\left(\frac{1}{1-z}\right) + \frac{1}{2(1-z)}, & \quad \text{for $p=\frac{1}{2}$}.
		\end{cases}
	\end{equation}
	The results stated in the lemma follow instantly from \eqref{eqn:binaryinc_M1z_sol} by taking into account \eqref{eqn:binaryinc_M1z_def} and extracting coefficients.
\end{proof}

\begin{lemma}\label{lem:binaryinc_momDn_asympt}
  Assume that $\frac{1}{2} < p < 1$. Then the $r$-th integer moments of $D_{n}$ have, for $r \ge 1$ fixed and $n \to \infty$, the following asymptotic behaviour:
	\begin{equation*}
	  \mathbb{E}(D_{n}^{r}) \sim \frac{2p-1}{p^{2}} \cdot \frac{r!}{\Gamma(r(2p-1)+1)} \cdot \left(\frac{p^{2}}{(2p-1)^{2}} n^{2p-1}\right)^{r}.
	\end{equation*}
\end{lemma}
\begin{proof}
  We introduce the functions
	\begin{equation*}
	  \tilde{M}_{r}(z) := \left.\frac{\partial^{r}}{\partial v^{r}} \tilde{Q}(z,v)\right|_{v=1},
	\end{equation*}
	with $\tilde{Q}(z,v)$ defined in Proposition~\ref{prop:binaryinc_Qzv}. 
	These functions are of interest due to the relation
	\begin{equation}\label{binaryinc_Mrz_relation}
	  \tilde{M}_{r}(z) = p^{2} \sum_{n \ge 1} \mathbb{E}\big((D_{n}+1)^{\underline{r}}\big) z^{n}, \quad \text{for $r \ge 2$}.
	\end{equation}
	According to the definition it further holds $\tilde{M}_{0}(z) = \frac{p}{1-z}$, whereas $\tilde{M}_{1}(z)$ has been already stated in \eqref{eqn:binaryinc_M1z_sol}.
	Differentiating \eqref{eqn:DEQ_Qzv_binaryinc} $r$ times w.r.t.\ $v$ and evaluating at $v=1$ yields for $r \ge 2$ the differential equation
	\begin{equation*}
	  \tilde{M}_{r}'(z) = \frac{2p}{1-z} \tilde{M}_{r}(z) + \sum_{1 \le \ell \le r-1} \binom{r}{\ell} \tilde{M}_{\ell}(z) \tilde{M}_{r-\ell}(z), \quad \tilde{M}_{r}(0)=0.
	\end{equation*}
	The solution of this first order linear differential equation can be obtained by standard means and is given as follows:
	\begin{equation}\label{eqn:binaryinc_Mrz_sol}
	  \tilde{M}_{r}(z) = \frac{1}{(1-z)^{2p}} \int_{0}^{z} (1-t)^{2p} \sum_{1 \le \ell \le r-1} \binom{r}{\ell} \tilde{M}_{\ell}(t) \tilde{M}_{r-\ell}(t) dt, \quad \text{for $r \ge 2$}.
	\end{equation}
	From \eqref{eqn:binaryinc_M1z_sol} and \eqref{eqn:binaryinc_Mrz_sol} it immediately follows by induction that the unique dominant singularity of the functions $\tilde{M}_{r}(z)$, $r \ge 1$, is at $z=1$. Furthermore, the following local behaviour in a complex neighbourhood of $z=1$ can be shown also by induction:
	\begin{equation}\label{eqn:binarying_alphar_asympt}
	  \tilde{M}_{r}(z) \sim \frac{\alpha_{r}}{(1-z)^{r(2p-1)+1}}, \quad \text{for $r \ge 1$},
	\end{equation}
	with certain constants $\alpha_{r}$. Namely, from the explicit solution \eqref{eqn:binaryinc_M1z_sol} it follows $\alpha_{1}=\frac{p^{2}}{2p-1}$, whereas for $r \ge 2$ we obtain by plugging the induction hypothesis into \eqref{eqn:binaryinc_Mrz_sol} and taking into account singular integration:
	\begin{equation*}
	  \tilde{M}_{r}(z) \sim \frac{1}{(1-z)^{r(2p-1)+1}} \cdot \frac{1}{(r-1)(2p-1)} \sum_{1 \le \ell \le r-1} \binom{r}{\ell} \alpha_{\ell} \alpha_{r-\ell},
	\end{equation*}
	which yields
	\begin{equation}\label{eqn:binarying_alphar_rec}
	  \alpha_{r} = \frac{1}{(r-1)(2p-1)} \sum_{1 \le \ell \le r-1} \binom{r}{\ell} \alpha_{\ell} \alpha_{r-\ell}, \quad r \ge 2.
	\end{equation}
	To treat recurrence \eqref{eqn:binarying_alphar_rec} we introduce the generating function $A(z) := \sum_{r \ge 1} \alpha_{r} \frac{z^{r}}{r!}$, which yields
	\begin{equation*}
	  z A'(z) - A(z) = \frac{1}{2p-1} A(z)^{2}.
	\end{equation*}
	The solution of this differential equation, which satisfies the initial condition $A'(0)= \alpha_{1} = \frac{p^{2}}{2p-1}$, is given as follows as can be checked easily:
	\begin{equation*}
	  A(z) = \frac{p^{2} (2p-1) z}{(2p-1)^{2} -p^{2}z}.
	\end{equation*}
	Thus the coefficients $\alpha_{r} = r! [z^{r}] A(z)$ are given by
	\begin{equation}\label{eqn:binaryinc_alphar_exp}
	  \alpha_{r} = \frac{r! p^{2r}}{(2p-1)^{2r-1}}.
	\end{equation}
	
	Combining \eqref{binaryinc_Mrz_relation}, \eqref{eqn:binarying_alphar_asympt} and \eqref{eqn:binaryinc_alphar_exp} and applying basic singularity analysis we obtain the asymptotic behaviour of the $r$-th integer moments of $D_{n}$, for $r \ge 2$:
	\begin{equation*}
	  \mathbb{E}(D_{n}^{r}) \sim \mathbb{E}(D_{n}^{\underline{r}}) \sim \mathbb{E}\big((D_{n}+1)^{\underline{r}}\big) = \frac{1}{p^{2}} [z^{n}] \tilde{M}_{r}(z) \sim \frac{2p-1}{p^{2}} \cdot \frac{r!}{\Gamma(r(2p-1)+1)} \cdot \left(\frac{p^{2}}{(2p-1)^{2}} \cdot n^{2p-1}\right)^{r}.
	\end{equation*}
	Due to Lemma~\eqref{lem:binarinc_expDn} this asymptotic result also holds for $r=1$ and thus finishes the proof.
\end{proof}

\begin{proof}[Proof of Theorem~\ref{the:binaryinc_thm}]
  According to Lemma~\ref{lem:binaryinc_probDn_limit} it holds that $\mathbb{P}\{D_{n}=m\} \to p_{m}$, for $m \ge 1$ fixed and $n \to \infty$, with numbers $p_{m}$ given there. Summing up the $p_{m}$ yields (by taking in mind the generating function of the Catalan numbers) the total mass
	\begin{equation*}
	  w := \sum_{m \ge 1} p_{m} = \sum_{m \ge 1} \frac{1}{m+1} \binom{2m}{m} p^{m-1} (1-p)^{m+1}
		= \frac{1-p}{p} \left(\frac{1-\sqrt{(1-2p)^{2}}}{2p(1-p)}-1\right).
	\end{equation*}
	For $0 < p \le \frac{1}{2}$ this yields $w=1$, thus the values $p_{m}$ characterize the discrete limit $D$ of $D_{n}$ as stated in the theorem. 
	
	Contrary, for $\frac{1}{2} < p < 1$ we obtain $w=\big(\frac{1-p}{p}\big)^{2} < 1$, which indicates that the limit contains also  a non-discrete part; the mass missing is $1-w = \frac{2p-1}{p^{2}}$. According to Lemma~\ref{lem:binaryinc_momDn_asympt} we obtain in this case for the $r$-th integer moments the asymptotic behaviour
	\begin{equation*}
	  \mathbb{E}\left(\Big(\frac{(2p-1)^{2}}{p^{2} n^{2p-1}} D_{n}\Big)^{r}\right) \sim \frac{2p-1}{p^{2}} \cdot \frac{r!}{\Gamma(r(2p-1)+1)}, \quad \text{for $r \ge 1$}.
	\end{equation*}
	Since the values on the right hand side are the $r$-th integer moments of a $\text{Mittag-Leffler}(2p-1)$ distributed r.v.\ multiplied with the factor $\frac{2p-1}{p^{2}}$ an application of the theorem of Fr\'{e}chet and Shohat proves the stated limiting distribution result.
\end{proof}

\section{Proof of Theorem~\ref{thm:ExpLengthRandomPath_b-ary} concerning the $b$-ary model}

\begin{proof}
  Due to symmetry reasons it also holds for the $b$-ary model that $L_{n}$ is distributed as the length $L_{n}^{[L]}$ of the leftmost source-to-sink path in a random size-$n$ series-parallel network. According to the description via bucket recursive trees the length of the leftmost path is one plus the sum of the lengths of the leftmost paths in the subblocks corresponding to the trees of the first forest, i.e., the forest attached to label $1$; see Figure~\ref{fig:link_decomposition_buckettree_network_bary}. Using the formal description \eqref{eqn:formaleqn_bbucket} of bucket recursive trees and introducing the generating function
\begin{equation*}
  F(z,v) := \sum_{n \ge 1} \sum_{m \ge 1} T_{n} \mathbb{P}\{L_{n} = m\} \frac{z^{n}}{n!} v^{m} = \sum_{n \ge 1} \sum_{m \ge 1} \mathbb{P}\{L_{n}=m\} \frac{z^{n}}{n} v^{m},
\end{equation*}
an application of the symbolic method yields the following description of the problem via a $b$-th order non-linear differential equation:
\begin{equation*}
  F^{(b)}(z,v) = (b-1)! v e^{F(z,v)} e^{(b-1) T(z)} = \frac{(b-1)! v}{(1-z)^{b-1}} \cdot e^{F(z,v)},
\end{equation*}
with initial conditions $F(0,v) = 0$, $F^{(k)}(0,v) = (k-1)! \cdot v$, for $1 \le k \le b-1(b)$.
It is slightly easier to consider the derivative (which would be also advantageous when studying higher moments):
\begin{equation*}
  G(z,v) := F'(z,v) = \sum_{n \ge 1} \sum_{m \ge 1} \mathbb{P}\{L_{n}=m\} z^{n-1} v^{m},
\end{equation*}
thus satisfying
\begin{equation}\label{eqn:Gzv_DEQ_bary}
\begin{split}
  G^{(b)}(z,v) & = \left(\frac{b-1}{1-z} + G(z,v)\right) \cdot G^{(b-1)}(z,v),\\
	G^{(k)}(0,v) & = k! \cdot v, \quad \text{for $0 \le k \le b-1$}.
\end{split}
\end{equation}

In order to get the expectation we introduce
\begin{equation*}
  E(z) := \left.\frac{\partial}{\partial v} G(z,v)\right|_{v=1} = \sum_{n \ge 1} \mathbb{E}(L_{n}) z^{n-1};
\end{equation*}
moreover, we use that $G(z,1) = T'(z) = \frac{1}{1-z}$. Differentiating \eqref{eqn:Gzv_DEQ_bary} w.r.t. $v$ and evaluating at $v=1$ yields the following $b$-th order homogeneous Eulerian differential equation for $E(z)$:
\begin{equation}\label{eqn:Ez_DEQ_bary}
\begin{split}
  E^{(b)}(z) & = \frac{b}{1-z} E^{(b-1)}(z) + \frac{(b-1)!}{(1-z)^{b}} E(z),\\
	E^{(k)}(0) & = k!, \quad \text{for $0 \le k \le b-1$}.
\end{split}
\end{equation}
To find the general solution we apply the Ansatz $E(z) = \frac{1}{(1-z)^{\lambda +1}}$; plugging it into the differential equation \eqref{eqn:Ez_DEQ_bary} leads after simple manipulations to the characteristic equation
\begin{equation}\label{eqn:cheqn_bary}
  P(\lambda)=0, \quad \text{with} \quad P(\lambda) := \lambda^{\overline{b}} - (b-1)!.
\end{equation}
Next we collect and sketch the proof of important facts concerning the roots of the characteristic equation.
\begin{itemize}
\item In the interval $[0, \infty)$ there exists exactly one real root, let us denote it by $\lambda_{1}$, which satisfies $\lambda_{1} \in (0,1)$: according to the definition, $P(\lambda)$ is a strictly increasing function on $[0, \infty)$. Furthermore $P(0)=-(b-1)! < 0$ and $P(1) = b! - (b-1)! > 0$, thus there exists a uniquely defined positive real root, which lies in the interval $(0,1)$.
\item For $b$ odd, in the interval $(-\infty, -(b-1)]$ there are no real roots: $\lambda^{\overline{b}}$ is negative, thus $P(\lambda) < 0$, for $\lambda$ in this interval.
\item For $b$ even, in the interval $(-\infty, -(b-1)]$ there exists exactly one real root: it holds $\lambda^{\overline{b}} = (-\lambda) \cdot (-\lambda -1) \cdots (-\lambda - (b-1)) = (-\lambda -(b-1))^{\overline{b}}$. Thus, when defining $\mu := - \lambda - (b-1)$ and $\tilde{P}(\mu) := \mu^{\overline{b}}-(b-1)!$, the function $\tilde{P}(\mu)$ is strictly increasing for $\mu \ge 0$ with a uniquely defined root for $\mu \in (0,1)$. Equivalently, when $\lambda \le -(b-1)$, there is exactly one real root for $\lambda \in (-b, -(b-1))$.
\item In the interval $(-(b-1),0)$ there are no real roots: let us assume $\lambda \in [-t, -(t-1)]$, with $t \in \{1, 2, \dots, b-1\}$. Then, elementary term-by-term estimates show the inequality
\begin{equation*}
  |\lambda^{\overline{b}}| = |\lambda| \cdot |\lambda + 1| \cdots |\lambda + b-1| < t! \cdot (b-t)! \le (b-1)!,
\end{equation*}
which implies that $P(\lambda) = \lambda^{\overline{b}} - (b-1)! < 0$, for $\lambda$ in this interval.
\item All roots of the characteristic equation are simple: it is sufficient to show that all $b-1$ zeros of the derivative $P'(\lambda)$ of the characteristic polynomial are located in the real interval $(-(b-1),0)$, since $P(\lambda)$ does not have zeros there. Differentiating $P(\lambda)$ gives the following expression (which could be simplified, but for our purpose this form is advantageous):
\begin{equation*}
  P'(\lambda) = \sum_{j=0}^{b-1} \prod_{0 \le k \le b-1, k \neq j} (\lambda + k).
\end{equation*}
When evaluating $P'(\lambda)$ for $\lambda = - t$, with $t \in \{0, 1, \dots, b-1\}$, one obtains after simple manipulations
\begin{equation*}
  P'(-t) = (-1)^{t} \cdot t! \cdot (b-1-t)!.
\end{equation*}
Thus, there are $b-1$ real intervals $(-(b-1),-(b-2))$, $(-(b-2),-(b-3))$, \dots, $(-1,0)$, where $P'(\lambda)$ has a sign-change and thus where it must have a zero; since the polynomial $P'(\lambda)$ has degree $b-1$, all zeros are real and are located in the stated interval.
\item The uniquely determined positive real root $\lambda_{1}$ has the largest real part amongst all roots: let us consider $\lambda' \in \mathbb{C}$, $\lambda' \neq \lambda_{1}$, with $\Re(\lambda') \ge \lambda_{1}$, thus $\lambda' = \lambda_{1} + \alpha + i \beta$, with $\alpha \ge 0$ and $(\alpha, \beta) \neq (0,0)$. Since $\lambda_{1} > 0$ and $\alpha \ge 0$ it holds $|\lambda' + k| > \lambda_{1}+k$, for $0 \le k \le b-1$, and thus $|(\lambda')^{\overline{b}}| > \lambda_{1}^{\overline{b}} = (b-1)!$. The triangle inequality shows then 
\begin{equation*}
  |P(\lambda')| \ge ||(\lambda')^{\overline{b}}| - (b-1)!| > 0,
\end{equation*}
i.e., $\lambda'$ is not a root of the characteristic equation.
\end{itemize}

Summarizing, the characteristic equation \eqref{eqn:cheqn_bary} has $b$ different roots $\lambda_{1}, \dots, \lambda_{b}$, which satisfy $\lambda_{1} > \Re(\lambda_{j})$, for $j \ge 2$, with $\lambda_{1} \in (0,1)$ the uniquely determined positive real root.
As an immediate consequence, we get that the general solution of the differential equation \eqref{eqn:Ez_DEQ_bary} is given as follows:
\begin{equation}\label{eqn:bary_Ez_sol}
  E(z) = \sum_{i=1}^{b} \frac{\beta_{i}}{(1-z)^{\lambda_{i}+1}},
\end{equation}
with coefficients $\beta_{i} \in \mathbb{C}$. When adapting the solution \eqref{eqn:bary_Ez_sol} to the initial conditions given in \eqref{eqn:Ez_DEQ_bary} one obtains that the coefficients $\beta_{i}$, $1 \le i \le b$, are characterized via the following system of linear equations:
\begin{equation}\label{eqn:LEQ_bary}
  \sum_{i=1}^{b} \beta_{i} \cdot (\lambda_{i}+1)^{\overline{k}} = k!, \quad 0 \le k \le b-1.
\end{equation}
Simple expressions for the solution of \eqref{eqn:LEQ_bary} can be obtained by adapting the (somewhat lengthy) computations given in \cite{Panholzer2003,Panholzer2004}. Namely, it turns out that the coefficients $\beta_{i}$ are given as follows:
\begin{equation}\label{eqn:coeff_betai_bary}
  \beta_{i} = \frac{1}{1+\lambda_{i} \cdot (H_{\lambda_{i}+b-1}-H_{\lambda_{i}})}, \quad \text{for $1 \le i \le b$}.
\end{equation}

However, it is sufficient for the proof of the theorem to show that the coefficients stated in \eqref{eqn:coeff_betai_bary} indeed solve the linear equations \eqref{eqn:LEQ_bary}, which will be done next.
For this purpose we give an alternative representation of the expressions given in \eqref{eqn:coeff_betai_bary}. 
We start with the linear factorization of the characteristic polynomial
\begin{equation*}
  P(\lambda) = \lambda^{\overline{b}} - (b-1)! = (\lambda -\lambda_{1}) \cdots (\lambda - \lambda_{b}),
\end{equation*}
and consider the derivative:
\begin{equation*}
  P'(\lambda) = \lambda^{\overline{b}} \cdot \sum_{k=0}^{b-1} \frac{1}{\lambda+k} = \sum_{p=1}^{b} \prod_{\ell \neq p} (\lambda - \lambda_{\ell}).
\end{equation*}
Evaluating at $\lambda = \lambda_{i}$ yields
\begin{equation*}
\begin{split}
  \prod_{\ell \neq i} (\lambda_{i} - \lambda_{\ell}) & = P'(\lambda_{i}) = \lambda_{i}^{\overline{b}} \cdot (H_{\lambda_{i}+b-1}-H_{\lambda_{i}-1}) = \lambda_{i}^{\overline{b}} \cdot \big(H_{\lambda_{i}+b-1}-H_{\lambda_{i}}+\frac{1}{\lambda_{i}}\big)\\
	& = (\lambda_{i}+1)^{\overline{b-1}} \cdot \big(1+\lambda_{i} (H_{\lambda_{i}+b-1}-H_{\lambda_{i}})\big),
\end{split}
\end{equation*}
thus showing the product form
\begin{align*}
  \beta_{i} & = \frac{1}{1+\lambda_{i}(H_{\lambda_{i}+b-1}-H_{\lambda_{i}})}\\
	& = \frac{(\lambda_{i}+1)^{\overline{b-1}}}{\prod_{\ell \neq i} (\lambda_{i}-\lambda_{\ell})} = \frac{(\lambda_{i}+1)^{\overline{b-1}} (-1)^{b-i} \prod_{1 \le j < \ell \le b, j,\ell \neq i} (\lambda_{\ell}-\lambda_{j})}{\prod_{1 \le j < \ell \le b}(\lambda_{\ell}-\lambda_{j})},
\end{align*}
where the last expression follows after simple manipulations. In order to prove validity of the linear equations \eqref{eqn:LEQ_bary} we have to simplify the following sums, for $0 \le k \le b-1$:
\begin{equation*}
\begin{split}
  S_{k} & := \sum_{i=1}^{b} \beta_{i} \cdot (\lambda_{i}+1)^{\overline{k}}\\
	& = \frac{1}{\prod_{1 \le j < \ell \le b}(\lambda_{\ell}-\lambda_{j})} \cdot \sum_{i=1}^{b} (\lambda_{i}+1)^{\overline{b-1}} (\lambda_{i}+1)^{\overline{k}} (-1)^{b-i} \prod_{1 \le j < \ell \le b, j,\ell \neq i} (\lambda_{\ell}-\lambda_{i}).
\end{split}
\end{equation*}
To this aim we give an interpretation of these sums via determinants (obtained by expanding the last row and using the factorization of the Vandermonde determinant):
\begin{equation*}
  S_{k} = \frac{1}{\prod\limits_{1 \le j < \ell \le b}(\lambda_{\ell}-\lambda_{j})} \cdot \underbrace{\left| \begin{array}{cccc} 1 & 1 & \dots & 1\\ \lambda_{1} & \lambda_{2} & \dots & \lambda_{b}\\ \lambda_{1}^{2} & \lambda_{2}^{2} & \dots & \lambda_{b}^{2}\\ \vdots & \vdots & \ddots & \vdots\\ \lambda_{1}^{b-2} & \lambda_{2}^{b-2} & \dots & \lambda_{b}^{b-2}\\ (\lambda_{1}+1)^{\overline{b-1}} (\lambda_{1}+1)^{\overline{k}} & (\lambda_{2}+1)^{\overline{b-1}} (\lambda_{2}+1)^{\overline{k}} & \dots & (\lambda_{b}+1)^{\overline{b-1}} (\lambda_{b}+1)^{\overline{k}}\end{array}\right|}_{=: \Lambda_{k}}.
\end{equation*}
Apparently,
\begin{equation*}
  (\lambda+1)^{\overline{k}} = (\lambda+1) \cdots (\lambda+k) = k! + \lambda \cdot Q_{k-1}(\lambda),
\end{equation*}
with $Q_{k-1}(\lambda)$ a certain polynomial in $\lambda$ of degree $k-1$ (or the zero polynomial if $k=0$). This yields the following simplification for the entries of the last row in $\Lambda_{k}$:
\begin{equation*}
   (\lambda_{i}+1)^{\overline{b-1}} (\lambda_{i}+1)^{\overline{k}} = k! (\lambda_{i}+1)^{\overline{b-1}} + \lambda_{i}^{\overline{b}} Q_{k-1}(\lambda_{i}) = k! (\lambda_{i}+1)^{\overline{b-1}} + (b-1)! Q_{k-1}(\lambda_{i}),
\end{equation*}
since $\lambda_{i}^{\overline{b}} = (b-1)!$. Thus, the entries in the last row of $\Lambda_{k}$ are polynomials in $\lambda_{i}$ of degree $b-1$. By elementary transformations of the first $b-1$ rows of $\Lambda_{k}$ all coefficients of powers $\le b-2$ can be annihilated and it remains a multiple of the Vandermonde determinant:
\begin{equation*}
  \Lambda_{k} = \left|\begin{array}{cccc} 1 & 1 & \dots & 1\\ \lambda_{1} & \lambda_{2} & \dots & \lambda_{b}\\ \vdots & \vdots & \ddots & \vdots\\
	\lambda_{1}^{b-2} & \lambda_{2}^{b-2} & \dots & \lambda_{b}^{b-2}\\ k! \lambda_{1}^{b-1} & k! \lambda_{2}^{b-1} & \dots & k! \lambda_{b}^{b-1} \end{array}\right| = k! \cdot \prod_{1 \le j < \ell \le b} (\lambda_{\ell}-\lambda_{j}).
\end{equation*}
Thus, indeed $S_{k} = k!$, for $0 \le k \le b-1$, which finishes the proof.
\end{proof}

\end{appendix}

\end{document}